\documentclass[11pt]{article}

\usepackage{amsmath}
\usepackage{amssymb}
\usepackage{amsthm}
\usepackage{mathabx}
\usepackage{pgfplots}
\usepackage{tikz}
\usetikzlibrary{trees}
\usepackage{hyperref}
\usepackage[T2A]{fontenc}
\usepackage[utf8]{inputenc}
\usepackage{authblk}
\usepackage{datetime}
\DeclareMathOperator\supp{supp}

\newfont{\bbf}{msbm10 scaled\magstep1}

\def\arxivprefix{https://arxiv.org}

\newcounter{glob}[subsection]
\renewcommand\theglob{%
	\ifnum\arabic{section}=0\else\arabic{section}.\fi %
	\ifnum\arabic{subsection}=0\else\arabic{subsection}.\fi %
	\arabic{glob}}

\newtheorem{thm}[glob]{Theorem}
\newtheorem{lemma}[glob]{Lemma}
\newtheorem{cor}[glob]{Corollary}

\theoremstyle{definition}\newtheorem{defi}[glob]{Definition}
\theoremstyle{remark}\newtheorem{remark}[glob]{Remark}
\theoremstyle{remark}

\def\Z{\mathbb{Z}}
\def\A{\mathcal{A}}
\def\C{\mathbb{C}}
\def\R{\mathbb{R}}
\def\Q{\mathbb{Q}}
\def\N{\mathbb{N}}

\def\P{\mathbb{P}^1}
\def\G{\widetilde{G}}
\def\Pl{PSL_2(\Z)}

\newcommand*{\email}[1]{
	\normalsize\href{mailto:#1}{#1}\par}
\newcommand{\keywords}[1]{\textbf{\textit{Keywords---}} #1}

\pgfplotsset{width=8cm}

\newdateformat{UKvardate}{%
	\THEDAY\ \monthname[\THEMONTH], \THEYEAR}
\UKvardate

\title{\bf Non-triviality of the Poisson boundary of random walks on the group $H(\Z)$ of Monod}
\author{Bogdan Stankov}
\affil{D\'epartement de math\'ematiques et applications, \'Ecole normale sup\'erieure, CNRS,\\ PSL Research University, 75005 Paris, France\\ \email{bogdan.zl.stankov@gmail.com}}
\date{\today}

\begin{document}
\maketitle
 
\begin{abstract}
We give sufficient conditions for the non-triviality of the Poisson boundary of random walks on $H(\Z)$ and its subgroups. The group $H(\Z)$ is the group of piecewise projective homeomorphisms over the integers defined by Monod. For a finitely generated subgroup $H$ of $H(\Z)$, we prove that either $H$ is solvable, or every measure on $H$ with finite first moment that generates it as a semigroup has non-trivial Poisson boundary. In particular, we prove the non-triviality of the Poisson boundary of measures on Thompson's group $F$ that generate it as a semigroup and have finite first moment, which answers a question by Kaimanovich.
\end{abstract}


\section{Introduction}
\thispagestyle{empty}
In 1924 Banach and Tarski~\cite{Banach-Tarski-original} decompose a solid ball into five pieces, and reassemble them into two balls using rotations. That is now called the Banach-Tarski paradox. Von Neumann~\cite{Neumann1929} observes that the reason for this phenomenon is that the group of rotations of $\R^3$ admits a free subgroup. He introduces the concept of amenable groups. Tarski~\cite{Tarski1938} later proves amenability to be the only obstruction to the existence of "paradoxical" decompositions (like the one in Banach-Tarski's article~\cite{Banach-Tarski-original}) of the action of the group on itself by multiplication, as well as any free actions of the group. One way to prove the result of Banach-Tarski is to see it as an almost everywhere free action of $SO_3(\R)$ and correct for the countable set where it is not (see e.g. Wagon~\cite[Cor.~3.10]{Banach-Tarski}).

The original definition of amenability of a group $G$ is the existence of an invariant mean. A mean is a normalised positive linear functional on $l^\infty(G)$. It is called invariant if it is preserved by translation on the argument. Groups that contain free subgroups are non-amenable. It is proven by Ol'shanskii in 1980~\cite{Olshanskii1980} that it is also possible for a non-amenable group to not have a free subgroup. Adyan~\cite{MR682486} shows in 1982 that all Burnside groups of a large enough odd exponent (which are known to be infinite by result of Novikov and Adyan from 1968~\cite{adyannovakov}) are non-amenable. Clearly they do not contain free subgroups. For more information and properties of amenability, see~\cite{bartholdi},\cite{article},\cite{greenleaf},\cite{Banach-Tarski}.

It is worth noting that despite the existence of a large amount of equivalent definitions of amenability, to our knowledge until recently all examples of non-amenable groups without free subgroups are proven (Ol'shanskii~\cite{Olshanskii1980}, Adyan~\cite{MR682486}, Ol'shanskii~\cite{0036-0279-35-4-L13}, Ol'shanskii-Sapir~\cite{finpresnonam}) to be such using the co-growth criterion. See Grigorchuk~\cite{Gri77} for the announcement of the criterion, or~\cite{grigcogreng} for a full proof. For other proofs, see Cohen~\cite{cogr1}, Szwarc~\cite{cogr3}. The criterion is closely related to Kesten's criterion in terms of probability of return to the origin~\cite{kesten}.

Monod constructs in \cite{h-main} a class of groups of piecewise projective homeomorphisms $H(A)$ (where $A$ is a subring of $\R$). By comparing the action of $H(A)$ on the projective line $\P(\R)$ with that of $PSL_2(A)$, he proves that it is non-amenable for $A\neq\Z$ and without free subgroups for all $A$. This can be used to obtain non-amenable subgroups with additional properties. In particular, Lodha~\cite{Lodha2014} proves that a certain subgroup of $H(\Z[\frac{\sqrt{2}}{2}])$ is of type $F_\infty$ (in other words, such that there is a connected CW complex $X$ which is aspherical and has finitely many cells in each dimension such that $\pi_1(X)$ is isomorphic to the group). That subgroup was constructed earlier by Moore and Lodha~\cite{lhodafinpres} as an example of a group that is non-amenable, without free subgroup and finitely presented. It has three generators and only $9$ defining relations (compare to the previous example by Ol'shanskii-Sapir~\cite{finpresnonam} with $10^{200}$ relations). This subgroup is the first example of a group of type $F_\infty$ that is non-amenable and without a free subgroup. Later, Lodha~\cite{lhoda-tarski} also proves that the Tarski numbers (the minimal number of pieces needed for a paradoxical decomposition) of all the groups of piecewise projective homeomorphisms are bounded by $25$.

It is not known whether the group $H(\Z)$ of piecewise projective homeomorphisms in the case $A=\Z$ defined by Monod is amenable. One of the equivalent conditions for amenability is the existence of a non-degenerate measure with trivial Poisson boundary (see Kaimanovich-Vershik~\cite{kaimpoisson}, Rosenblatt~\cite{rosenblatt}). This measure can be chosen to be symmetric. It is also known that amenable groups can have measures with non-trivial boundary. In a recent result Frisch-Hartman-Tamuz-Vahidi-Ferdowski~\cite{choquet-deny} describe an algebraic necessary and sufficient condition for a group to admit a measure with non-trivial boundary. In the present paper we give sufficient conditions for non-triviality of the Poisson boundary on $H(\Z)$. There are several equivalent ways to define the Poisson boundary (see Kaimanovich-Vershik~\cite{kaimpoisson}). Consider a measure $\mu$ on a group $G$ and the random walk it induces by multiplication on the left. It determines an associated Markov measure $P$ on the trajectory space $G^\N$.

\begin{defi}\label{poisson}
Consider the following equivalence relation on $G^\N$: two trajectories $(x_0,x_1,\dots)$ and $(y_0,y_1,\dots)$ are equivalent if and only if there exist $i_0\in\N$ and $k\in\Z$ such that for every $i>i_0$ $x_i=y_{i+k}$. In other words, if the trajectories coincide after a certain time instant up to a time shift. The \textit{Poisson boundary} (also called \textit{Poisson-Furstenberg boundary}) of $\mu$ on $G$ is the quotient of $(G^\N,P)$ by the measurable hull of this equivalence relation.
\end{defi}

Note that if the support of the measure does not generate $G$, in which case we say that the measure is \textit{degenerate}, this defines the boundary on the subgroup generated by the support of the measure rather than on $G$. For a more recent survey on results concerning the Poisson boundary, see~\cite{Erschler2010}.

Kim, Koberda and Lodha have shown in~\cite{chainhomeo} that $H(\Z)$ contains Thompson's group $F$ as a subgroup. This group is the group of orientation-preserving piecewise linear self-isomorphisms of the closed unit interval with dyadic slopes, with a finite number of break points, all break points being dyadic numbers (see Cannon-Floyd-Perry~\cite{thomsoncfp} or Meier's book~\cite[Ch.~10]{meier} for details and properties). It is not known whether it is amenable, which is a celebrated open question. Kaimanovich~\cite{kaimanovichthompson} and Mishchenko~\cite{mischenko2015} prove that the Poisson boundary on $F$ is not trivial for finitely supported non-degenerate measures. They study the induced walk on the dyadic numbers in their proofs. However, there exist non-degenerate symmetric measures on $F$ for which the induced walk has trivial boundary as proven by Juschenko and Zheng~\cite{juszheng}. The results of the current article are inspired by the paper of Kaimanovich. It is not hard to prove that $H(\Z)$ is not finitely generated (see Remark~\ref{fingen}), so we will consider measures the support of which is not necessarily finite.

Our main result is as follows. Consider the group $H(\Z)$ of piecewise projective homeomorphisms, as defined by Monod~\cite{h-main}, in the case $A=\Z$. For $g\in H(\Z)$ denote by $Br(g)$ the number of \textit{break points} of $g$, which is the ends of pieces in its piecewise definition. We will say that a measure $\mu$ on a subgroup of $H(\Z)$ has \textit{finite first break moment} if the expected number of break points $\mathbb{E}[Br]$ is finite. A group $H$ is called \textit{locally solvable} if all finitely generated subgroups are solvable. Then

\begin{thm}\label{main}
For any subgroup $H$ of $H(\Z)$ which is not locally solvable and any measure $\mu$ on $H$ with finite first break moment $\mathbb{E}[Br]$ and such that the support of $\mu$ generates $H$ as a semigroup, the Poisson boundary of $(H,\mu)$ is non-trivial.
\end{thm}

For a measure $\mu$ on a finitely generated group, we say that $\mu$ has \textit{finite first moment} if the word length over any finite generating set has finite first moment with respect to $\mu$. This is well defined as word lengths over different finite generating sets are bilipschitz, and in particular the finiteness of the first moment does not depend on the choice of generating set. We remark (see Remark~\ref{brfin}) that any measure $\mu$ on a finitely generated subgroup $H$ of $H(\Z)$ that has finite first moment also has finite expected number of break points. Therefore by Theorem~\ref{main} if $\mu$ is a measure on a non-solvable finitely generated subgroup $H$ such that the support of $\mu$ generates $H$ as a semigroup and $\mu$ has finite first moment, the Poisson boundary of $(H,\mu)$ is non-trivial. Furthermore, in the other case we will show (Lemma~\ref{mineps}) that so long as $H$ is not abelian, we can construct a symmetric non-degenerate measure with finite $1-\varepsilon$ moment and non-trivial Poisson boundary.

The structure of the paper is as follows. In Section~\ref{prelim}, given a fixed $s\in\R$, to every element $g\in H(\Z)$ we associate (see Definition~\ref{confdef}) a configuration $C_g$. Each configuration is a function from the orbit of $s$ into $\Z$. The value of a configuration $C_g$ at a given point of the orbit of $s$ represents the slope change at that point of the element $g$ to which it is associated. There is a natural quotient map of the boundary on the group into the boundary on the configuration space. The central idea of the paper is to show that under certain conditions, the value of the configuration at a given point of the orbit of $s$ almost always stabilises. If that value is not fixed, this then implies non-triviality of the boundary on the configuration space, and thus non-triviality of the Poisson boundary on the group. These arguments bear resemblance to Kaimanovich's article on Thompson's group~\cite{kaimanovichthompson}, but we would like to point out that the action on $\R$ considered in the present article is different.

In Section~\ref{sectfour} we obtain the first result for non-triviality of the Poisson boundary (see Lemma~\ref{constr}). Measures satisfying the assumptions of that lemma do not necessarily have finite first break moment. In Section~\ref{thompsect} we study copies of Thompson's group $F$ in $H(\Z)$. Building on the results from it, in Section~\ref{schreier} we obtain transience results (see Lemma~\ref{algtho}) which we will need to prove Theorem~\ref{main}. In Section~\ref{anothersuff} we prove Lemma~\ref{ltwo} which is the main tool for proving non-triviality of the Poisson boundary. In the particular case of Thompson's group, the lemma already allows us to answer a question by Kaimanovich~\cite[7.A]{kaimanovichthompson}:

\begin{cor}\label{finfirstthomp}
Any measure on Thompson's group $F$ that has finite first moment and the support of which generates $F$ as a semi-group has non-trivial Poisson boundary.
\end{cor}

We mention that the arguments of Lemma~\ref{ltwo} could also be applied for the action and configurations considered in Kaimanovich's article, giving an alternative proof of the corollary. Combining the lemma with the transience results from Section~\ref{schreier} we obtain non-triviality of the Poisson boundary under certain conditions (see Lemma~\ref{algone}), which we will use to prove the main result. As the negation of those conditions passes to subgroups, it suffices to show that if $H$ is finitely generated and does not satisfy them, it is then solvable, which we do in Section~\ref{algsec}. Remark that the theorem generalises the result of Corollary~\ref{finfirstthomp}. In Section~\ref{last} we give an additional remark on the case of finite $1-\varepsilon$ moment.

%
%
%
%
\section*{Acknowledgements}

I would like to offer my special thanks to my thesis director, Anna Erschler. Discussions with Laurent Bartholdi have been extremely helpful, and they inspired me to consider the case of finite first moment. I would also like to thank Vadim Kaimanovich for his remarks on the preliminary version of the paper. I am also grateful to Dmytro Savchuk for the question that lead to Remark~\ref{grapham}.

\section{Preliminaries}\label{prelim1}

\subsection{$\Pl$ and $H(\Z)$}

The projective linear group $PSL_2(\R)$ is defined as $SL_2(\R)/\{Id,-Id\}$, which is the natural quotient that describes the linear actions on the projective space $\P(\R)$. As the latter can be defined as $\mathbb{S}/(x\sim-x)$, we can think of it as a circle for understanding the dynamics of the action of the projective group. Remark that it is commonly understood as the boundary of the hyperbolic plane. In this paper we will not be interested in the interior of the hyperbolic plane as we do a piecewise definition of $H(A)$ on $\P(\R)$. An element $h\in PSL_2(\R)$ is called:

\begin{enumerate}
\item \textbf{Hyperbolic} if $|tr(h)|>2$ (or equivalently, $tr(h)^2-4>0$). In this case a calculation shows that $h$ has two fixed points in $\P(\R)$. One of the points is attractive and the other repulsive for the dynamic of $h$, meaning that starting from any point and multiplying by $h$ (respectively $h^{-1}$) we get closer to the attractive (resp. the repulsive) fixed point.
\item \textbf{Parabolic} if $|tr(h)|=2$. In this case $h$ has exactly one "double" fixed point. We can identify $\P(\R)$ with $\R\cup\{\infty\}$ in such a way that the fixed point is $\infty$, in which case $h$ becomes a translation on $\R$. We will go into detail about the identification below.
\item \textbf{Elliptic} if $|tr(h)|<2$. Then $h$ has no fixed points in $\P(\R)$ and is conjugate to a rotation. If we consider it as an element of $PSL_2(\C)$, we can see that it has two fixed points in $\P(\C)$ that are outside $\P(\R)$.
\end{enumerate}

Consider an element $\begin{pmatrix}x\\y\end{pmatrix}\in\R^2\setminus0$. If $y\neq 0$, identify it with $\frac{x}{y}$, otherwise with $\infty$. This clearly passes on $\P(\R)$, and the action of $PSL_2(\R)$ becomes $\begin{pmatrix}a&b\\c&d\end{pmatrix}. x=\frac{ax+b}{cx+d}$. The conventions for infinity are $\begin{pmatrix}a&b\\c&d\end{pmatrix}(\infty)=\frac{a}{c}$ if $c\neq0$ and $\infty$ otherwise, and if $c\neq 0$, $\begin{pmatrix}a&b\\c&d\end{pmatrix}.(-\frac{d}{c})=\infty$. Note that by conjugation we can choose any point to be the infinity.

Let us now look into the groups defined by Monod~\cite{h-main}. We define $\Gamma$ as the group of all homeomorphisms of $\R\cup\{\infty\}$ that are piecewise in $PSL_2(\R)$ with a finite number of pieces. Take a subring $A$ of $\R$. We define $\Gamma(A)$ to be the subgroup of $\Gamma$ the elements of which are piecewise in $PSL_2(A)$ and the extremities of the intervals are in $P_A$, the set of fixed points of hyperbolic elements of $PSL_2(A)$. 

\begin{defi}
	The group of piecewise projective homeomorphisms $H(A)$ is the subgroup of $\Gamma(A)$ formed by the elements that fix infinity.
\end{defi}

It can be thought of as a group of homeomorphisms of the real line, and we will use the same notation in both cases. We will note $G=H(\Z)$ to simplify. Note in particular that $\infty\notin P_\Z$. This means that the germs around $+\infty$ and $-\infty$ are the same for every element of $G$. The only elements in $\Pl$ that fix infinity are 

\begin{equation}\label{agrp}
\left\{\left(\alpha_n=\begin{pmatrix}1&n\\0&1\end{pmatrix}\right)_{n\in\Z}\right\}= G\cap \Pl.
\end{equation}

Fix $g\in G$ and let its germ at infinity (on either side) be $\alpha_n$. Then $g\alpha_{-n}$ has finite support. The set of elements $\bar{G}\subset G$ that have finite support is clearly a subgroup, and therefore if we denote $\A=\{\alpha_n,n\in\Z\}$, we have

$$G=\bar{G}+\A$$

For the purposes of this article, we also need to define:

\begin{defi}\label{tildeg}
Consider the elements of $\Gamma$ that fix infinity and are piecewise in $\Pl$. We call the group formed by those elements the \textit{piecewise $\Pl$ group}, and denote it as $\G$.
\end{defi}

Remark that in an extremity $\gamma$ of the piecewise definition of an element $g\in\G$, the left and right germs $g(\gamma-0)$ and $g(\gamma+0)$ have a common fixed point. Then $g(\gamma+0)^{-1}g(\gamma-0)\in \Pl$ fixes $\gamma$. Therefore the extremities are in $P_\Z\cup\Q\cup\{\infty\}$, that is in the set of fixed points of any (not necessarily hyperbolic) elements of $\Pl$. In other words, the only difference between $\G$ and $G=H(\Z)$ is that $\G$ is allowed to have break points in $\Q\cup\{\infty\}$, that is in the set of fixed points of parabolic elements. Clearly, $G\leq\G$. This allows us to restrain elements, which we will need in Section~\ref{algsec}:

\begin{defi}\label{restr}
	Let $f\in\G$, and $a,b\in\R$ such that $f(a)=a$ and $f(b)=b$. The function $f\restriction_{(a,b)}$ defined by $f\restriction_{(a,b)}(x)=f(x)$ for $x\in(a,b)$ and $f(x)=x$ otherwise is called a restriction.
\end{defi}

Remark that $f\restriction_{(a,b)}\in\G$. The idea of this definition is that we extend the restrained function with the identity function to obtain an element of $\G$.

The subject of this paper is $G$, however in order to be able to apply results from previous sections in Section~\ref{algsec}, we will prove several lemma for $\G$. The equivalent result will easily follow for $G$ just from the fact that it is a subgroup.

\subsection{Random walks}

Throughout this article, for a measure $\mu$ on a group $H$ we will consider the random walk by multiplication on the left. That is the walk $(x_n)_{n\in\N}$ where $x_{n+1}=y_nx_n$ and the increments $y_n$ are sampled by $\mu$. In other words, it is the random walk defined by the kernel $p(x,y)=yx^{-1}$. Remark that for walks on groups it is standard to consider the walk by multiplications on the right. In this article the group elements are homeomorphisms on $\R$ and as such they have a natural action on the left on elements of $\R$, which is $(f,x)\mapsto f(x)$.

We will use Definition~\ref{poisson} as the definition of Poisson boundary. For completeness' sake we also mention its description in terms of harmonic functions. For a group $H$ and a probability measure $\mu$ on $H$ we say that a function $f$ on $H$ is \textit{harmonic} if for every $g\in H$, $f(g)=\sum_{h\in H}f(hg)\mu(h)$. For a non-degenerate measure, the $L^\infty$ space on the Poisson boundary is isomorphic to the space of bounded harmonic functions on $H$, and the exact form of that isomorphism is given by a classical result called the \textit{Poisson formula}. In particular, non-triviality of the Poisson boundary is equivalent to the existence of non-trivial bounded harmonic functions.

%

We recall the entropy criterion for triviality of the Poisson boundary.

\begin{defi}
	Consider two measures $\mu$ and $\lambda$ on a discrete group $H$. We denote $\mu*\lambda$ their \textit{convolution}, defined as the image of their product by the multiplication function. Specifically:
	
	$$\mu*\lambda(A)=\int\mu(Ah^{-1})d\lambda(h).$$
\end{defi}

Remark that $\mu^{*n}$ gives the probability distribution for $n$ steps of the walk, starting at the neutral element. For a probability measure $\mu$ on a countable group $H$ we denote $H(\mu)$ its \textit{entropy}, defined by

$$H(\mu)=\sum_{g\in H}-\mu(g)\log{\mu(g)}.$$

One of the main properties of entropy is that the entropy of a product of measures is not greater than the sum of their entropies. Combining that with the fact that taking image of a measure by a function does not increase its entropy, we obtain $H(\mu*\lambda)\leq H(\mu) +H(\lambda)$. Avez~\cite{avez72} introduces the following definition:

\begin{defi}
	The \textit{entropy of random walk} (also called \textit{asymptomatic entropy}) of a measure $\mu$ on a group $H$ is defined as $\lim_{n\rightarrow\infty}\frac{H(\mu^{*n})}{n}$.
\end{defi}

\begin{thm}[Entropy Criterion (Kaimanovich-Vershik~\cite{kaimpoisson}, Derriennic~\cite{derast})]\label{entropy}
Let $H$ be a countable group and $\mu$ a non-degenerate probability measure on $H$ with finite entropy. Then the Poisson boundary of $(H,\mu)$ is trivial if and only if the asymptotic entropy of $\mu$ is equal to zero.
\end{thm}

\section{Some properties of groups of piecewise projective homeomorphisms}\label{prelim}

In Subsection~\ref{slopechange} we study $P_\Z$ and the group action locally around points of it. In Subsection~\ref{confsect}, using the results from the first subsection, to each element $g\in\G$ we associate a configuration $C_g$. We then also describe how to construct an element with a specific associated configuration.

\subsection{Slope change points in $G=H(\Z)$}\label{slopechange}

Let $g$ be a hyperbolic element of $\Pl$. Let it be represented by $\begin{pmatrix}a&b\\c&d\end{pmatrix}$ and denote $tr(g)=a+d$ its trace. Then its fixed points are $\frac{d-a\pm\sqrt{tr(g)^2-4}}{c}$. As the trace is integer and greater than $2$ in absolute value, this number is never rational. Furthermore, it is worth noting that $\Q(\sqrt{tr(g)^2-4})$ is stable by $\Pl$ and therefore by $\G$ (and $G$). If we enumerate all prime numbers as $(p_i)_{i\in\N}$, we have, for $I\neq J\subset\N$ finite, $\Q(\sqrt{\prod_{i\in I}p_i})\cap\Q(\sqrt{\prod_{i\in J}p_i})=\Q$. We just mentioned that $P_\Z\cap\Q=\emptyset$ so we have

$$P_\Z=\bigsqcup_{I\subset\N\mbox{ finite}}\left(P_\Z\bigcap\Q\left(\sqrt{\prod_{i\in I}p_i}\right)\right)$$
where each set in the decomposition is stable by $\G$. Note also that the fixed points of parabolic elements of $\Pl$ are rational. This actually completely characterizes the set $P_\Z$, as we will now show that $P_\Z\bigcap\Q\left(\sqrt{\prod_{i\in I}p_i}\right)=\Q\left(\sqrt{\prod_{i\in I}p_i}\right)\setminus\Q$:

\begin{lemma}\label{all}
	Take any $s\in\Q(\sqrt{k})\setminus\Q$ for some $k\in\N$. Then $s\in P_\Z$.
\end{lemma}

Remark that $k$ is not an exact square, as $\Q(\sqrt{k})\setminus\Q$ has to be non-empty.

\begin{proof}
Note first that to have $\sqrt{tr^2-4}\in\Q(\sqrt{k})$ for some matrix it suffices to find integers $x\geq 2$ and $y$ such that $x^2-ky^2=1$. Indeed, any matrix with trace $2x$ will then satisfy this, for example $\begin{pmatrix}x&x^2-1\\1&x\end{pmatrix}$. This is known as Pell's equality, and has infinitely many solutions for any $k$ that is not a square (see Mordell's book~\cite[Ch.~8]{mordell}).

Write $s=\frac{p}{q}+\frac{p'}{q'}\sqrt{k}$ for some integers $p,q,p',q'$. Applying Pell's equality for $(p'q'q^2)^2k$, we obtain integers $x$ and $a$ such that $x^2-a^2(p'q'q^2)^2k=1$. In other words, $x^2-y^2k=1$ for $y=p'q'q^2a$. We construct $\begin{pmatrix}x+q'^2pqa&b\\q'^2q^2a&x-q'^2pqa\end{pmatrix}$ where $b=\frac{x^2-q'^4p^2q^2a^2-1}{q'^2q^2a}=p'^2q^2ak-q'^2p^2a\in a\Z$. The matrix has $s$ for a fixed point, and $s$ is not rational, therefore the matrix is a hyperbolic element of $\Pl$.
%
%
\end{proof}

\begin{remark}\label{fingen}
	The break points a finite number of elements of $H(\Z)$ are all contained in the sets $\Q(\sqrt{k})$ for a finite number of $k$, so Lemma~\ref{all} implies that $H(\Z)$ is not finitely generated.
\end{remark}

In order to define configurations, we wish to study the slope changes at elements of $P_\Z$. Consider $g\in\G$ and $s\in P_\Z$ such that $g(s+0)\neq g(s-0)$. Then it is easy to see that $f=g(\gamma-0)^{-1}g(\gamma+0)\in \Pl$ fixes $s$. Therefore, in order to study the slope changes we need to understand the stabiliser of $s$ in $\Pl$. We prove:

\begin{lemma}\label{cyclic}
	Fix $s\in\P(\R)$. The stabiliser $St_s$ of $s$ in $\Pl$ is either isomorphic to $\Z$ or trivial.
\end{lemma}

\begin{proof}
Assume that $St_s$ is not trivial, and let $f\in St_s$ be different from the identity. Clearly, $f$ is not elliptic. If $f$ is hyperbolic, $s\in P_\Z$, and if $f$ is parabolic, $s\in\Q\cup\{\infty\}$. We distinguish three cases, that is $s\in P_\Z$, $s=\infty$ and $s\in\Q$.
	
We first assume $s\in P_\Z$. Let $s=r+r'\sqrt{k}$ with $r,r'\in\Q$ and $k\in\Z$. Note that the calculations in the beginning of the section yield that for every element $f$ in $St_s$ that is not the identity, $f$ is hyperbolic and the other fixed point of $f$ is $\bar{s}=r-r'\sqrt{k}$. Let $i=\begin{pmatrix}\frac{1}{2}&-\frac{r+r'\sqrt{k}}{2}\\\frac{1}{r'\sqrt{k}}&1-\frac{r}{r'\sqrt{k}}\end{pmatrix}\in PSL_2(\R)$ and consider the conjugation of $St_s$ by $i$. By choice of $i$ we have that $i(s)=0$ and $i(\bar{s})=\infty$. Therefore the image of $St_s$ is a subgroup of the elements of $PSL_2(\R)$ that have zeros on the secondary diagonal. Furthermore, calculating the image of an example matrix $\begin{pmatrix}a&b\\c&d\end{pmatrix}$, for $tr=a+d$ the trace of the matrix, we get

\begin{equation}\label{cyc}
i\begin{pmatrix}a&b\\c&d\end{pmatrix}i^{-1}=\begin{pmatrix}\frac{\sqrt{tr^2-4}+tr}{2}&0\\0&\frac{\sqrt{tr^2-4}-tr}{2}\end{pmatrix}
\end{equation}

Thus to understand the image of $St_s$ we just need to study the elements of the form $\frac{x+y\sqrt{k}}{2}$ with $x^2-ky^2=4$. This appears in a generalized form of Pell's equation, and those elements are known~\cite[Ch.~8]{mordell} to be powers of a fundamental solution (which is also true for the classic Pell equation if you identify a solution $x^2-y^2k=1$ with a unit element $x+y\sqrt{k}$ in $\Z[\sqrt{k}]$). This proves that the image of $St_s$ by this conjugation, which is isomorphic to $St_s$, is a subgroup of a group isomorphic to $\Z$. $St_s$ is then also isomorphic to $\Z$. The matrix with the fundamental solution in the upper left corner defines a canonical generator for the group of elements of the form seen in (\ref{cyc}), and its smallest positive power in the image of $St_s$ defines a canonical generator for $St_s$.

Assume now $s=\infty$. As we described in (\ref{agrp}), the stabiliser of $\infty$ is $(\alpha_n)_{n\in\N}$, which is trivially isomorphic to $\Z$.

Lastly, assume that $s=\frac{p}{q}\in\Q$ with $p$ and $q$ co-prime. There exist $m$ and $n$ such that $pm+qn=1$. Then $i=\begin{pmatrix}m&n\\-q&p\end{pmatrix}\in \Pl$ verifies $i(s)=\infty$. Thus the conjugation by $i$ defines an injection from the subgroup that fixes $s$ into $St_\infty=\A$. We observe that non-trivial subgroups of $\Z$ are isomorphic to $\Z$, which concludes the proof.\end{proof}

Having an isomorphism between $St_s$ (for $s\in P_\Z$) and $\Z$ will be useful to us, so we wish to know its exact form. We prove:

\begin{lemma}\label{log}
Let $s\in P_\Z$. There exists $\phi_s\in\R^+$ that remains constant on the orbit $Gs$ of $s$ such that $f\mapsto\log_{\phi_s}(f'(s)))$ defines an isomorphism between $St_s$ and $\Z$.
\end{lemma}

\begin{proof}
The derivative on the fixed point is multiplicative. Therefore for a fixed $s$, this follows from Lemma~\ref{cyclic} and the fact that subgroups of $\Z$ are isomorphic to $\Z$ (or trivial, which is impossible here). What we need to prove is that $\phi$ remains constant on $Gs$. Fix $s$ and consider $s'\in Gs$. Let $j\in \Pl$ be such that $j(s)=s'$. Then the conjugation by $j$ defines a bijection between $St_s$ and $St_{s'}$. Calculating the derivative on an element $f\in St_s$ we get $(jfj^{-1})'(s')=j'(s)(j^{-1})'(j(s))f'(s)=f'(s)$, which proves the result.
\end{proof}

We further denote $\psi:\A\mapsto\Z$ (see \ref{agrp}) the map that associates $n$ to $\alpha_n$, and $\psi_r$ the conjugate map for any $r\in\Q$. Remark that this is well defined by Lemma~\ref{cyclic} and conjugations in $\Z$ being trivial. 

\subsection{Configurations}\label{confsect}

Fix $s\in P_\Z$ and let $\phi=\phi_s$ be given by Lemma~\ref{log}. By the isomorphism it defines, there exists an element $g_s$ that fixes $s$, such that $g_s'(s)=\phi_s$. As $s\notin\Q$, $g_s$ is hyperbolic. We associate to each element of the piecewise $\Pl$ group $\G$ (see Definition~\ref{tildeg}) a configuration representing the changes of slope at each point of the orbit $\G s=Gs$ of $s$, precisely:

\begin{defi}\label{confdef}To $g\in\G$ we assign $C_g:Gs\rightarrow\Z$ by

$$C_g(\gamma)=\log_\phi(g'(\gamma+0)g'(\gamma-0)^{-1}).$$
\end{defi}

Note that by choice of $\phi$ this value is well defined: indeed, $g(\gamma+0)g(\gamma-0)^{-1}\in \Pl$, fixes $\gamma$, and is therefore in $St_\gamma$.

Remark that by definition of $\G$ each configuration in the image of the association has a finite support. Remark also that the configuration ignores information about the changes in slope outside the orbit of $s$. For $s\in\Q$ we further denote $C_g(\gamma)=\psi_\gamma(g'(\gamma+0)g'(\gamma-0)^{-1})$, which will have similar properties. In the rest of the paper we will consider $s\in P_\Z$ unless otherwise specified. For completeness' sake, remark also that $G=H(\Z)\leq\G$ and the orbits of $G$ and $\G$ on $s$ are the same (as they are both the same as the orbit of $\Pl$) and therefore Definition~\ref{confdef} could be done directly for $G$, and what we would obtain is the same as restraining from the current definition.

\begin{lemma}\label{unit}\label{hs}
	For every $s\in P_\Z$, there exists an element $h_s\in G$ such that $h_s(s-0)^{-1}h_s(s+0)=g_s$ and all other slope changes of $h_s$ are outside $Gs$. In particular, $C_{h_s}=\delta_s$.
\end{lemma}

\begin{proof}
Fix $s\in P_\Z$ and let $k=k_s$ be the unique square-free integer such that $s\in\Q(\sqrt{k})$. We will construct $h_s$ such that $h_s(s)=s$. Note that in that case we have $C_{h_s^{-1}}=-\delta_s$. This implies that if we construct an element $\tilde{h}_s$ that verifies $\tilde{h}_s(s-0)^{-1}\tilde{h}_s(s+0)=g_s^{\pm1}$ and all other slope changes are outside $Gs$, choosing $h_s=\tilde{h}_s^{\pm1}$ gives the result. In other words, we can replace $g_s$ with $g_s^{-1}$. Seen as a function on $\R$, $g_s$ is defined in all points but $-\frac{d}{c}$. It is then continuous in an interval around $s$. Moreover, if the interval is small enough, $s$ is the only fixed point in it. Therefore for some $\varepsilon$, either $g_s(x)>x$ for every $x\in(s,s+\varepsilon)$, or $g_s(x)<x$ in that interval. As we have the right to replace it with its inverse, without loss of generality we assume that $g_s$ is greater than the identity in a right neighbourhood of of $s$.

Write $s=r+r'\sqrt{k}$ with $r,r'\in\Q$. Then the other fixed point of $g_s$ is its conjugate $\bar{s}=r-r'\sqrt{k}$. Remark that it is impossible for $-\frac{d}{c}$ to be between $s$ and $s'$ as the function $g_s$ is increasing where it is continuous and has the same limits at $+\infty$ and $-\infty$ (see Figure~\ref{trivplot}). If $r'<0$, $g_s$ is greater than the identity in $(s,\bar{s})$ as it is continuous there. In that case, it is smaller than the identity to the left of the fixed points, but as it is increasing and has a finite limit at $-\infty$, this implies (see Figure~\ref{trivplot}) that $-\frac{d}{c}<s$. Similarly, if $s>\bar{s}$, $g_s$ is increasing and greater than the identity to the right of $s$, but has a finite limit at $+\infty$, so $-\frac{d}{c}>s$.

\begin{figure}
	\centering
	\begin{minipage}{8cm}\centering\caption{Graphs of $g_s$ and the identity}\label{trivplot}
		\begin{tikzpicture}
		\begin{axis}[xmin=-4,xmax=4,ymin=-4,ymax=4,axis lines = middle, legend pos = south west,xtick={-10},ytick={17}]
		\addplot[domain=-3.8:3.8,color=black]{x};
		\addlegendentry{$Id$}
		\addplot[color=blue,samples=100,domain=-4:-2,restrict y to domain=-4:4,dashed,thick]{(2*x+3)/(x+2)};
		\addlegendentry{$g_s$}
		\addplot[color=red,samples=100,domain=-4:2,restrict y to domain=-4:4]{(2*x-3)/(2-x)};
		\addplot[color=red,samples=100,domain=2:4,restrict y to domain=-4:4]{(2*x-3)/(2-x)};
		\addlegendentry{$g_s^{-1}$}
		\addplot[color=blue,samples=100,domain=-2:4,restrict y to domain=-4:4,dashed,thick]{(2*x+3)/(x+2)};
		\node[label={-30:{$s$}},circle,fill,inner sep=1pt] at (axis cs:-1.732,-1.732) {};
		\end{axis}
		\end{tikzpicture}
	\end{minipage}
	\begin{minipage}{8cm}\centering\caption{Graphs of $g_s$ and $j_s$}\label{plot}
		\begin{tikzpicture}
		\begin{axis}[xmin=-4,xmax=4,ymin=-4,ymax=4,axis lines = middle, legend pos = south west,xtick={-10},ytick={17}]
		\addplot[domain=-3.8:3.8,color=black]{x};
		\addlegendentry{$Id$}
		\addplot[color=blue,samples=100,domain=-4:-2,restrict y to domain=-4:4,dashed,thick]{(2*x+3)/(x+2)};
		\addlegendentry{$g_s$}
		\addplot[color=red,samples=100,domain=-4:0,restrict y to domain=-4:4]{(2*x-3)/(2-x)};
		\addlegendentry{$g_s^{-1}$}	
		\addplot[samples=100,domain=0:4,restrict y to domain=-4:4,densely dotted,thick]{(4*x-1)/x};
		\addlegendentry{$j_s$}
		\addplot[color=blue,samples=100,domain=-2:4,restrict y to domain=-4:4,dashed,thick]{(2*x+3)/(x+2)};
		\addplot[color=red,samples=150,domain=0:2,restrict y to domain=-4:4]{(2*x-3)/(2-x)};
		\addplot[color=red,samples=100,domain=2:4,restrict y to domain=-4:4]{(2*x-3)/(2-x)};
		\addplot[samples=100,domain=-4:0,restrict y to domain=-4:4,densely dotted,thick]{(4*x-1)/x};
		\node[label={-1:{$\bar{t}$}},circle,fill,inner sep=1pt] at (axis cs:0.268,0.268) {};
		\node[label={110:{$\tilde{s}$}},circle,fill,inner sep=1pt] at (axis cs:0.414,1.5858) {};
		\end{axis}
		\end{tikzpicture}	
	\end{minipage}
\end{figure}

We will find a hyperbolic element $j_s$ verifying: the larger fixed point $t$ of $j_s$ is not in $Gs$ and $t>-\frac{d}{c}$, while the smaller fixed point $\bar{t}$ is between $s$ and $\bar{s}$, and $j_s$ is greater than the identity between $\bar{t}$ and $t$. If $r'<0$ consider the interval $(\bar{t},\bar{s})$. At its infimum, $j_s$ has a fixed point while $g_s$ is greater than the identity, and at its supremum the inverse is true. By the mean values theorem, there exists $\tilde{s}$ in that interval such that $j_s(\tilde{s})=g_s(\tilde{s})$ (see Figure~\ref{plot}). If $r'>0$, consider the interval $(s,-\frac{d}{c})$. At its infimum, $g_s$ is fixed and therefore smaller than $j_s$, and at its supremum $g_s$ diverges towards $+\infty$ while $j_s$ has a finite limit. Again by the mean values theorem, there exists $\tilde{s}$ in that interval where $g_s$ and $j_s$ agree. As $-\frac{d}{c}<t$ by hypothesis, in both cases we have $s<\tilde{s}<t$. We then define

\begin{equation*}h_s(x)=\begin{cases}
x & x\leq s \\
g_s(x) & s\leq x\leq\tilde{s}\\
j_s(x) & \tilde{s}\leq x\leq t \\
x & t\leq x \\
\end{cases}
\end{equation*}

Thus it would suffice to prove that we can construct $j_s$ that verifies those properties and such that $\tilde{s}\notin Gs$. Note that $\tilde{s}$ is a fixed point of $g_s^{-1}j_s$, so to prove that it is not in $Gs$ it will suffice to study the trace of the latter. Remark that in this definition $h_s$ is strictly greater than the identity in an open interval, and equal to it outside (this is with the assumption on $g_s$, in the general case $h_s$ has its support in an open interval, and is either strictly greater then the identity on the whole interval, or strictly smaller).

Write $r=\frac{p}{q}$. By Bezout's identity, there are integers $\tilde{m}$ and $\tilde{n}$ such that $q\tilde{n}-p\tilde{m}=1$. Then the matrix $i=\begin{pmatrix}\tilde{n}&p\\\tilde{m}&q\end{pmatrix}\in \Pl$ verifies $i.0=\frac{p}{q}$. Taking $\tilde{j}_s=i^{-1}j_si$ it suffices to find $\tilde{j}_s$ with fixed points outside $Gs$, the smaller one being close enough to $0$, and the greater one large enough. Remark that the only information we have on $g_s$ is its trace, so this does not complicate the computations for $\tilde{s}$.

We will define $\tilde{j}_s$ in the form $\begin{pmatrix}x'+ma'&n^2l_sa'-m^2a'\\a'&x'-ma'\end{pmatrix}$ where $x'^2-n^2a'^2l_s=1$. Its fixed points are $m\pm n\sqrt{l_s}$. By choosing $m$ arbitrarily large, the second condition will be satisfied. Note $ig_s^{-1}i^{-1}=\begin{pmatrix}\tilde{a}&\tilde{b}\\\tilde{c}&\tilde{d}\end{pmatrix}$ and $tr(g_s)^2-4=o^2k$. Calculating the trace of $g_s^{-1}j_s$ we get $tr(g_s)x'+a'\tilde{b}+mz_1+nz_2$ with $z_1,z_2\in\Z$. Then, admitting that $n$ divides $x'-1$ (which will be seen in the construction of $x'$) we obtain for some $z_i\in Z$, $i\in\N$:

\begin{equation}\label{moche}
\begin{split}
tr(g_s^{-1}j_s)^2-4&=mz_3+nz_4+a'^2\tilde{b}^2+2a'\tilde{b}x'tr(g_s)+x'^2tr(g_s)^2-tr(g_s)^2+tr(g_s)^2-4\\
&=mz_3+nz_5+a'^2\tilde{b}^2+2a'\tilde{b}tr(g_s)+n^2a'^2l_str(g_s)^2+o^2k\\
&=mz_3+nz_6+a'^2\tilde{b}^2+2a'\tilde{b}tr(g_s)+o^2k.
\end{split}
\end{equation}

Take a prime $p_s$ that is larger than $k$ and $b(tr(g_s)+2)$. There is an integer $a''<p_s$ such that $b(tr(g_s)+2)a''\equiv-1\mod{p_s}$. Take $a=o^2ka''$. Then 

$$a'^2\tilde{b}^2+2a'\tilde{b}tr(g_s)+o^2k=o^2k(b(tr(g_s)+2)a''+1)(b(tr(g_s)-1)).$$

As $\Z[p_s]$ is a field, clearly $b(tr(g_s)-2)a''\not\equiv-1\mod{p_s}$. As $b(tr(g_s)+2)a''<p_s^2$, the product is divisible by $p_s$ but not $p_s^2$. We will choose $m$ and $n$ divisible by $p_s^2$, which will then ensure that the value in (\ref{moche}) is divisible by $p_s$ but not $p_s^2$, proving that $\tilde{s}\notin Gs$.

All that is left is choosing $n$ and $m$. As we just noted, we need them to be multiples of $p_s^2$. Aside from that $n$ needs to satisfy $x'^2-n^2a'^2l_s=1$, $l_s$ must not be a square times $k$ and we need to be able to make $m-n\sqrt{l_s}$ arbitrarily small. Write $m=p_s^2m'$ and $n=p_s^2n'$. Then $m'$ can be anything so long as $m-n\sqrt{l_s}$ becomes arbitrarily small. In other words, we are only interested in the fractional part of $n'\sqrt{l_s}$. We choose $x'=n'^2a'^2p_s^5-1$ and will prove that the conditions are satisfied for $n'$ large enough. Then $x'^2-n^2a'^2l_s=1$ is satisfied for $l_s=p_s(n'^2a'^2p_s^5-2)$. In particular, $p_s$ divides $l_s$ but its square does not, so $l_s$ is not equal to a square times $k$. Moreover, $\sqrt{l_s}=\sqrt{(n'a'p_s^3)^2-2p_s}$ and as the derivative of the square root is strictly decreasing, $\sqrt{(n'a'p_s^3)^2-2p_s}-n'a'p_s^3\rightarrow0$ for $n'\rightarrow\infty$. Its factorial part then clearly converges towards $1$, which concludes the proof.\end{proof}

For a product inside the group $\G$, by the chain rule we have

$$(g_2g_1)'(\gamma)=g_2'(g_1(\gamma))g_1'(\gamma)$$
and thus

\begin{equation}\label{der}C_{g_2g_1}(\gamma)=C_{g_1}(\gamma)+C_{g_2}(g_1(\gamma))\end{equation}

That gives us a natural action of $\G$ on $\Z^{Gs}$ by the formula $(g,C)\rightarrow C_g+S^gC$ where $S^gC(\gamma)=C(g(\gamma))$. It is easy to check that it also remains true for $s\in\Q$.

\begin{lemma}\label{nostable}
	There is no configuration $C:Gs\rightarrow\Z$ such that $C=C_{h_s}+S^{h_s}C$.
\end{lemma}

Indeed, applying (\ref{der}) and taking the value at $s$ we get a contradiction. 

Consider $g$ and $h$ such $C_g=C_h$. We have $C_{Id}=C_{g^{-1}}+S^{g^{-1}}C_g$ and thus $C_{hg^{-1}}=C_{g^{-1}}+S^{g^{-1}}C_h=C_{Id}=0$. We denote

$$H_s=\{g\in G:C_g=0\}.$$

Then:

\begin{lemma}\label{generate}
	The element $h_s$ and the subgroup $H_s$ generate $G$ for every $s\in P_\Z$.	
\end{lemma}

\begin{proof}We show for $g\in G$ by induction on $\|C_g\|_1=\sum_{x\in Gs}|C_g(x)|$ that it is in the group generated by $\{h_s\}\cup H_s$ . The base is for $\|C_g\|_1=0$, in which case we have $C_g\equiv 0$ and the result is part of the statement hypothesis. We take $g\in G$ and assume that every element with smaller $l^1$ measure of its configuration is in the group generated by $\{h_s\}\cup H_s$. We take any $\alpha\in\supp(C_g)$. Without loss of generality, we can assume that $C_g(\alpha)>0$. As $g(\alpha)\in Gs$, by Lemma~\ref{trace} there exists $h\in H_s$ such that $h(s)=g(\alpha)$ and $C_h=0$. Let $\tilde{g}=hh_sh^{-1}$. As $h_s\in\{h_s\}\cup H_s$, we have $\tilde{g}\in\langle \{h_s\}\cup H_s\rangle$. Applying the composition formula~(\ref{der}) we obtain $C_{\tilde{g}}(x)=0$ for $x\neq g(\alpha)$ and $C_{\tilde{g}}(g(\alpha))=1$. We consider $\bar{g}=\tilde{g}^{-1}g$. If $x\neq g(\alpha)$, by the composition formula (\ref{der}) we get $C_{\bar{g}}(x)=C_g(x)$, and at $\alpha$ we have $C_{\bar{g}}(\alpha)=C_g(\alpha)-1$. By hypothesis we then have $\bar{g}\in\langle \{h_s\}\cup H_s\rangle$, and as $\tilde{g}$ is also included in this set, so is $g$.\end{proof}

\begin{lemma}\label{trace}
	For any $g\in \Pl$ and $\gamma\in\R$ there exists $h\in H_s$ such that $g(\gamma)=h(\gamma)$.
\end{lemma}

\begin{proof}By Monod's construction in~\cite[Proposition~9]{h-main}, we know that we can find $h\in G$ that agrees with $g$ on $\gamma$ of the form $q^{-1}g$ where $q=\begin{pmatrix}a&b+ra\\c&d+rc\end{pmatrix}$ in the interval between its fixed points that contains infinity and the identity otherwise. To have this result, what is required is that either $r$ or $-r$ (depending on the situation) be large enough. Clearly, $C_h\equiv0$ would follow from slope change points of $q$ being outside $Gs$ (as neither of them is infinity). In particular, it is enough to prove that for some infinitely large $r$, the fixed points of $\begin{pmatrix}a&b+ra\\c&d+rc\end{pmatrix}$ are outside $\Q(\sqrt{k})$. The trace of that matrix is $(a+d)+rc$. Let $p$ be a large prime number that does not divide $2$, $k$ or $c$. As $c$ and $p$ are co-prime, there exists $r_0$ such that $a+d+r_0c=p+2\pmod{p}$. Then for every $i\in\Z$, we have $(a+d+(r_0+p^2i)c)^2-4=4p(modp^2)$. As $p$ and $4$ are co-prime, this implies that for each $r=r_0+p^2i$ the fixed points of that matrix are not in $\Q(\sqrt{k})$ as $p$ does not divide $k$.\end{proof}

\section{Convergence condition}\label{sectfour}

Fix $s\in P_\Z$ and let us use the notations from Subsection~\ref{confsect}. For a measure $\mu$ on $\G$ we denote $C_\mu=\bigcup_{g\in\supp(\mu)}\supp(C_g)$ its "support" on $Gs$. That is, $C_\mu\subset Gs$ is the set of points in which at least one element that is inside the support of $\mu$ in the classical sense changes slope. We thus obtain the first result

\begin{lemma}\label{base}
	Consider the piecewise $\Pl$ group $\G$ (see Definition~\ref{tildeg}). Let $\mu$ be a measure on a subgroup of $\G$ such that $C_\mu$ is transient with respect to $\mu$ for the natural action of $\G$ on $\R$ and $h_s$ is in the semigroup generated by $\supp(\mu)$. Then the Poisson boundary of $\mu$ on the subgroup is not trivial.
\end{lemma}

\begin{proof}Consider a random walk $g_n$ with $g_{n+1}=h_ng_n$. For a fixed $\gamma\in Gs$ we have

$$C_{g_{n+1}}(\gamma)=C_{g_n}(\gamma)+C_{h_n}(g_n(\gamma))$$

By the hypothesis of transiency this implies that $C_{g_n}(\gamma)$ stabilises. In other words, $C_{g_n}$ converges pointwise towards a limit $C_\infty$. This defines a hitting measure on $\Z^{Gs}$ that is a quotient of $\mu$'s Poisson boundary. Moreover, it is $\mu$-invariant by the natural action on $\Z^{Gs}$. It remains to see that it is not trivial. Assume the opposite, which is that there exists a configuration $C$ such that for almost all walks, the associated configuration $C_{g_n}$ converges pointwise to $C$. By hypothesis there are elements $h_1,\dots,h_m$ with positive probability such that $h_mh_{m-1}\dots h_1=h_s$. There is a strictly positive probability for a random walk to start with $h_mh_{m-1}\dots h_1$. Applying~(\ref{der}) we get $C=C_{h_s}+S^{h_s}C$, which is contradictory to Lemma~\ref{nostable}.\end{proof}

This lemma, along with Lemma~\ref{generate} implies:

\begin{lemma}\label{constr}
Fix $s\in P_\Z$. Let $\mu$ be a measure on $G=H(\Z)$ that satisfies the following conditions:
		
	(i) The element $h_s$ belongs to the support of $\mu$,
	
	(ii) The intersection of the support of $\mu$ with the complement of $H_s$ is finite,
	
	(iii) The action of $\mu$ on the orbit of $s$ is transient.
		
Then the Poisson boundary of $\mu$ is non-trivial.
\end{lemma}

We will now show how measures satisfying whose assumptions can be constructed. Remark that the question of existence of a measure with non-trivial boundary has already been solved by Frisch-Hartman-Tamuz-Vahidi-Ferdowski~\cite{choquet-deny}. In our case, notice that $\A\subset H_s$ (see (\ref{agrp})), and it is isomorphic to $\Z$. We can then use a measure on $\A$ to ensure transience of the induced walk on the orbit. To prove that, we use a lemma from Baldi-Lohoué-Peyrière~\cite{var} (see also Woess~\cite[Section~2.C,3.A]{woess2000random}). Here we formulate a stronger version of the lemma, as proven by Varopoulos~\cite{Varopoulis1983}:

\begin{lemma}[Comparison lemma]\label{var}
	Let $P_1(x,y)$ and $P_2(x,y)$ be doubly stochastic kernels on a countable set $X$ and assume that $P_2$ is symmetric. Assume that there exists $\varepsilon\geq 0$ such that
	
	$$P_1(x,y)\geq\varepsilon P_2(x,y)$$
	for any $x,y$. Then
	
	\begin{enumerate}
		\item For any $0\leq f\in l^2(X)$
		
		$$\sum_{n\in\N}\langle P_1^nf,f\rangle\leq \frac{1}{\varepsilon}\sum_{n\in\N}\langle P_2^nf,f\rangle.$$
		
		\item If $P_2$ is transient then so is $P_1$ (for any point $x\in X$, it follows from (1) applied to $f=\delta_x$).
	\end{enumerate}
\end{lemma}

Here, doubly stochastic kernels means that the operators are reversible and the inverse is also Markov. It is in particular the case for $P(x,y)=\mu(yx^{-1})$ for some measure on a group (as the inverse is $(x,y)\mapsto\mu(xy^{-1})$).

\begin{remark}\label{gen}
If $\lambda$ is a transient measure on $\A$ and $\mu$ satisfies conditions (i) and (ii) of Lemma~\ref{constr}, then the comparison lemma by Baldi-Lohoué-Peyrière (Lemma~\ref{var}) implies that $\varepsilon\lambda+(1-\varepsilon)\mu$ satisfies all the conditions of the lemma for any $0<\varepsilon<1$. In other words, this is a way to construct non-degenerate symmetric measures on $G$ with non-trivial Poisson boundary.
\end{remark}

For completeness' sake, we show that there exist measures positive on all of $G$ that have non-trivial boundary.

\begin{lemma}
	Let $\mu$ be a measure on a group $H$ with finite entropy and non-zero asymptotic entropy and which generates $H$ as a semigroup. Then there exists a measure $\tilde{\mu}$ with support equal to $H$ that also has finite entropy and non-zero asymptotic entropy. Furthermore, if $\mu$ is symmetric, so is $\tilde{\mu}$.
\end{lemma}

\begin{proof}
Define $\tilde{\mu}=\frac{1}{e}\sum_{i\in\N}\frac{\mu^{*i}}{i!}$. By a result of Kaimanovich~\cite[Corollary~to~Theorem~4]{entrlemma} we get

$$h(H,\tilde{\mu})=h(H,\mu)\sum_{i\in\N}\frac{i}{ei!}=h(H,\mu).$$

Moreover, as the entropy of $\tilde{\mu}^{*n}$ is not smaller than the entropy of $\tilde{\mu}$, finite asymptotic entropy implies finite entropy.
%
%
%
%
\end{proof}

From this lemma and the entropy criterion Theorem~\ref{entropy} it follows that to have a measure positive on all of $G$ with non-trivial boundary it suffices to construct a measure verifying the conditions of Lemma~\ref{constr} with finite asymptotic entropy, which we can achieve with the construction presented in Remark~\ref{gen}.

\section{Thompson's group as a subgroup of $G=H(\Z)$}\label{thompsect}
In~\cite{chainhomeo} Kim, Kuberda and Lodha show that any two increasing homeomorphisms of $\R$ the supports of which form a 2-chain (as they call it) generate, up to taking a power of each, a group isomorphic to Thompson's group $F$. Let us give the exact definition of this term. For a homeomorphism $f$ of $\R$ we call its support $\supp(f)$ the set of points $x$ where $f(x)\neq x$. Remark that we do not define the closure of that set as support, as it is sometimes done. Consider four real numbers $a,b,c,d$ with $a<b<c<d$. Take two homeomorphisms $f$ and $g$ such that $\supp(f)=(a,c)$ and $\supp(g)=(b,d)$. In that case we say that their supports form a 2-chain, and the homeomorphisms generate a 2-prechain group. In other words, two homeomorphisms generate a 2-prechain if their supports are open intervals that intersect each other but neither is contained in the other.

Clearly, there exist many such pairs in $G$. We will give a simple example. Fix $s$ and find positive rational numbers $\tilde{r}$ and $\tilde{r}'$ such that $\tilde{r}<s<\tilde{r}+\tilde{r}'\sqrt{p_s}<t$. Recall that $p_s$ is a prime larger than $k$. Then choose a hyperbolic element $\tilde{g}$ that fixes $\tilde{r}+\tilde{r}'\sqrt{p_s}$ and define

\begin{equation*}\tilde{h}_s(x)=\begin{cases}
\tilde{g}_s(x) & \tilde{r}-\tilde{r}'\sqrt{p_s}\leq x\leq\tilde{r}+\tilde{r}'\sqrt{p_s} \\
x & \mbox{otherwise.} \\
\end{cases}
\end{equation*}

By definition of $\tilde{r}$ and $\tilde{r}'$, $\tilde{h}_s$ and $h_s$ clearly form a 2-prechain, and thus up to a power they generate a copy of Thompson's group (see~\cite[Theorem~3.1]{chainhomeo}). We will denote $\mathfrak{a}_s$ the action $F\curvearrowright\R$ this defines. To obtain the convergence results, we need to prove that the induced random walks on the Schreier graphs of certain points are transient. By the comparison lemma by Baldi-Lohoué-Peyrière (Lemma~\ref{var}) it would suffice to prove it for the simple random walk on the graph, which is why we will study its geometry. In the dyadic representation of Thompson's group, the geometry of the Schreier graph on dyadic numbers has been described by Savchuk~\cite[Proposition~1]{slav10}. It is a tree quasi-isometric to a binary tree with rays attached at each point (see Figure~\ref{sav}), which implies transience of the simple random walk. For a different proof of transience see Kaimanovich~\cite[Theorem~14]{kaimanovichthompson}. We will see that the Schreier graph has similar geometry in the case of $\mathfrak{a}_s$ (see Figure~\ref{treefig}).

\begin{lemma}\label{tree-old}
Consider two homeomorphisms $f$ and $g$ of $\R$ the supports of which are $\supp(f)=(a,c)$ and $\supp(g)=(b,d)$ with $a<b<c<d$. Denote $H$ the group generated by $f$ and $g$. Then the simple random walk on the Schreier graph of $H$ on the orbit of $b$ is transient. 
\end{lemma}

\begin{proof}
Up to replacing $f$ or $g$ with its inverse, we can assume without loss of generality that $f(x)>x$ for $x\in\supp(f)$ and $g(x)>x$ for $x\in\supp(g)$. Denote by $\Gamma$ the Schreier graph of $H$ on the orbit of $b$. The vertices of this graph are the points of the orbit $Hb$ of $b$ by $H$, and two points are connected by an edge if and only if $f$, $g$, $f^{-1}$ or $g^{-1}$ sends one point into the other. Denote by $\tilde{\Gamma}$ the subgraph defined by the vertexes that belong to the closed interval $[b,c]$. At every point $x$ of $\Gamma$ such that $x\notin[b,c]$, in a neighbourhood $(x-\varepsilon,x+\varepsilon)$ of $x$, one of the two elements $f$ and $g$ acts trivially, and the other one is strictly greater than the identity map. Without loss of generality, let $f$ act trivially. Let $i_0$ be the largest integer such that $g^{i_0}(x)\in[b,c]$. Then the set of points $(g^i(x))_{i\geq i_0}$ is a ray that starts at an element of $\tilde{\Gamma}$. As the simple random walk on $\Z$ is recurrent (see~\cite[Chapter~3,~Theorem~2.3]{durrett2005probability}), the walk always returns to $\tilde{\Gamma}$ in finite time, and that part of the graph ($\tilde{\Gamma})$ is what we need to study.

Replacing, if necessary, $f$ or $g$ by its power, we can assume that $g^{-1}(c)<f(b)$. Denote $A=[b,g^{-1}(c)]=g^{-1}([b,c])$, $B=[f(b),c]=f([b,c])$ and $C=(g^{-1}(c),f(b))=[b,c]\setminus(A\cup B)$. Consider $x\in\tilde{\Gamma}$ with $x\neq b$ and $x\notin C$. Consider a reduced word $c_nc_{n-1}\dots c_1$ with $c_i\in\{f^{\pm1},g^{\pm1}\}$ that describes a path in $\tilde{\Gamma}$ from $b$ to $x$. In other words $c_nc_{n-1}\dots c_1(b)=x$ and the suffixes of that word satisfy $c_ic_{i-1}\dots c_1(b)\in\tilde{\Gamma}$ for every $i\leq n$. The fact that the word is reduced means that $c_i\neq c_{i+1}^{-1}$ for every $i$. We claim that if $x\in A$, this word ends with $g^{-1}=c_n$, and if $x\in B$, $c_n=f$.

We prove the latter statement by induction on the length of the word $n$. If a word of length one, it is $g$ since $f$ fixes $b$ and since $g^{-1}(b)\notin [b,c]$. As $g(b)\in B$ this gives the base for the induction. 

Assume that the result is true for any reduced word of length strictly less than $n$ whose suffixes, when applied to $b$, stay in $[b,c]$. We will now prove it for $x=c_nc_{n-1}\dots c_1(b)$. We denote $y=c_{n-1}c_{n-2}\dots c_1(b)$ the point just before $x$ in that path. We first consider the case $x\in B$ (as we will see from the proof, the other case is equivalent). We distinguish three cases: $y\in A$, $y\in B$ and $y\in C$.

If $y\in A$, by induction hypothesis we have $c_{n-1}=g^{-1}$. As the word is reduced we thus have $c_n\neq g$. However, from $y\in A$ and $x\in B$ we have $y<x$. Therefore, $c_n\notin\{f^{-1},g^{-1}\}$, and the only possibility left is $c_n=f$.

If $y\in B$, by induction hypothesis we have $c_{n-1}=f$. Therefore, as the word is reduced, $c_n\neq f^{-1}$. From $g^{-1}(c)<f(b)$ it follows that $g(B)\cap[b,c]=\emptyset$. As $x\in B$, this implies that $c_n\neq g$. Similarly, $g^{-1}(B)\subset A$, therefore $c_n\neq g^{-1}$. The only possibility left is $c_n=f$.

If $y\in C$, consider the point $y'=c_{n-2}\dots c_1(b)$. If $y'\in A$, by induction hypothesis $c_{n-2}=g^{-1}$. Then $c_{n-1}\neq g$. As $y>y'$, this implies that $c_{n-1}=f$. However, $g(A)\subset B$, which is a contradiction. In a similar way, we obtain a contradiction for $y'\in B$. However, both $f^{-1}(C)$ and $g(C)$ are outside $[b,c]$, while $f(C)\subset B$ and $g^{-1}(C)\subset A$. Therefore the case $y\in C$ is impossible by induction hypotheses on $c_{n-2}\dots c_1$.

This completes the induction. Remark that we also obtained $\tilde{\Gamma}\cap C=\emptyset$, so the result holds for all points of $\tilde{\Gamma}$. In particular, if two paths in $\tilde{\Gamma}$ described by reduced words arrive at the same point, the last letter in those words is the same, which implies that $\tilde{\Gamma}$ is a tree. Remark also that the result implies that $c\notin\tilde{\Gamma}$ as $c\in B$ and $f^{-1}(c)=c$.

Moreover, for a vertex $x\in A$, we have that $f(x)$, $g(x)$ and $g^{-1}(x)$ also belong to $\tilde{\Gamma}$. Similarly, for $x\in B$, $g^{-1}(x)$, $f(x)$ and $f^{-1}(x)$ are in $\tilde{\Gamma}$. Therefore every vertex aside from $b$ has three different neighbours. The simple walk on $\tilde{\Gamma}$ is thus transient.
\end{proof}

By the comparison lemma by Baldi-Lohoué-Peyrière (Lemma~\ref{var}), this implies transience on the Schreier graph of $s$ for any measure on $G$ such that $h_s$ and $\bar{h}_s$ are in the semigroup generated by the support of the measure. If the support of a given measure generates $G$ as a semigroup, conditions $(i)$ and $(iii)$ in Lemma~\ref{base} are then automatically satisfied. In particular, any measure $\mu$ on $G$ that generates it as a semigroup and such that there exists $s$ for which $\supp(\mu)\cap(G\setminus H_s)$ is finite has a non-trivial Poisson boundary.

In the proof of Lemma~\ref{tree-old} we obtained a description of the graph of $\mathfrak{a}_s$, which is similar to the one by Savchuk~\cite{slav10} in the case of the dyadic action:

\begin{remark}\label{tree}
Consider two homeomorphisms $f$ and $g$ of $\R$ the supports of which are $\supp(f)=(a,c)$ and $\supp(g)=(b,d)$ with $a<b<c<d$. Denote $H$ the group generated by $f$ and $g$. Then the Schreier graph of $H$ on the orbit of $b$ is described in Figure~\ref{treefig} (solid lines are labelled by $f$ and dashed lines by $g$).

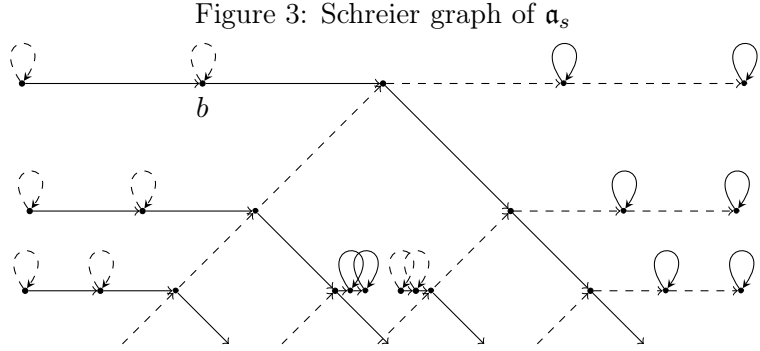
\begin{figure}[!h]\caption{Schreier graph of $\mathfrak{a}_s$}\label{treefig}\centering\begin{tikzpicture}[-stealth]
\tikzset{node/.style={circle,draw,inner sep=0.7,fill=black}}
\tikzset{every loop/.style={min distance=8mm,in=55,out=125,looseness=10}}
\tikzstyle{level 1}=[level distance=2.4cm,sibling distance=3cm]
\tikzstyle{level 2}=[level distance=2.4cm,sibling distance=12mm]
\tikzstyle{level 3}=[level distance=1.5cm,sibling distance=5mm]
\tikzstyle{level 4}=[level distance=1cm,sibling distance=5mm]

\node[node,label=below:{$b$}](0){}
child[grow=left,<-]{node[node](-1){}}
child[grow=right,->]{node[node]{}
	child[grow=south west,<-,dashed]{node[node]{}
		child[grow=south west,<-,dashed]{node[node]{}
			child[grow=south west,<-,dashed]{}
			child[grow=south east,->,solid]{}
			child[grow=left,<-,solid,level distance=1cm]{node[node](1){} child[grow=left,level distance=1cm]{node[node](8){}}}}
		child[grow=south east,->,solid]{node[node]{}
			child[grow=south west,<-,dashed]{}
			child[grow=south east,->,solid]{}
			child[grow=right,->,dashed,level distance=0.2cm]{node[node](9){} child[grow=right,level distance=0.2cm]{node[node](10){}}}}
		child[grow=left,<-,solid,level distance=1.5cm]{node[node](2){} child[grow=left,level distance=1.5cm]{node[node](11){}}}}
	child[grow=south east,->,solid]{node[node]{}
		child[grow=south west,<-,dashed]{node[node]{}
			child[grow=south west,<-,dashed]{}
			child[grow=south east,->,solid]{}
			child[grow=left,<-,solid,level distance=0.2cm]{node[node](3){} child[grow=left,level distance=0.2cm]{node[node](3b){}}}}
		child[grow=south east,->,solid]{node[node]{}
			child[grow=south west,<-,dashed]{}
			child[grow=south east,->,solid]{}
			child[grow=right,->,dashed,level distance=1cm]{node[node](4){} child[grow=right,level distance=1cm]{node[node](4b){}}}}
		child[grow=right,->,dashed,level distance=1.5cm]{node[node](6){} child[grow=right,level distance=1.5cm]{node[node](6b){}}}}
	child[grow=right,->,dashed,level distance=2.4cm]{node[node](5){} child[grow=right, level distance=2.4cm]{node[node](7){}}}
};
\draw (0) edge[loop above,dashed] (0);
\draw (-1) edge[loop above,dashed] (-1);
\draw (1) edge[loop above,dashed] (1);
\draw (2) edge[loop above,dashed] (2);
\draw (3) edge[loop above,dashed] (3);
\draw (3b) edge[loop above,dashed] (3b);
\draw (4) edge[loop above] (4);
\draw (4b) edge[loop above] (4b);
\draw (5) edge[loop above] (5);
\draw (6) edge[loop above] (6);
\draw (6b) edge[loop above] (6b);
\draw (7) edge[loop above] (7);
\draw (8) edge[loop above,dashed] (8);
\draw (9) edge[loop above] (9);
\draw (10) edge[loop above] (10);
\draw (11) edge[loop above,dashed] (11);

\end{tikzpicture}\end{figure}
\end{remark}

\begin{proof}
In the proof of Lemma~\ref{tree-old} we have shown that for every vertex $x\in\tilde{\Gamma}$ that is not $b$, $x$ has exactly three different neighbours in $\tilde{\Gamma}$. We also proved that $\tilde{\Gamma}$ is a tree. It is therefore a binary tree. Furthermore, if $x\in A$, it is equal to $g^{-1}(y)$ where $y$ is closer to $b$ than $x$ (in the graph), and if $x\in B$, $x=f(y)$ where $y$ is again closer to $b$. We think of $y$ as the parent of $x$. Then every vertex $x$ has two children: left child $g^{-1}(x)$ and right child $f(x)$. Furthermore, if $x$ is a left child, $x\in A$ and $f^{-1}(x)\notin\tilde{\Gamma}$. Equivalently, if $x$ is a right child, $g(x)\notin\tilde{\Gamma}$.\end{proof}

Compare to the Schreier graph of the dyadic action as described by Savchuk~\cite[Proposition~1]{slav10}(see Figure~\ref{sav}).

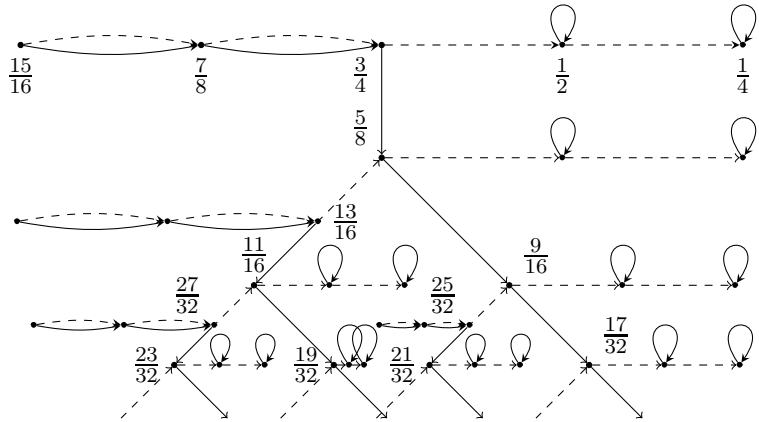
\begin{figure}[!h]\centering\caption{Schreier graph of the dyadic action of $F$ for the standard generators}\label{sav}\begin{tikzpicture}[-stealth]
\tikzset{no edge/.style={edge from parent/.append style={draw=none}}}
\tikzset{node/.style={circle,draw,inner sep=0.7,fill=black}}
\tikzset{every loop/.style={min distance=8mm,in=55,out=125,looseness=10}}
\tikzstyle{level 1}=[level distance=2.4cm,sibling distance=3cm]
\tikzstyle{level 2}=[level distance=2.4cm,sibling distance=12mm]
\tikzstyle{level 3}=[level distance=1.5cm,sibling distance=5mm]
	\tikzstyle{level 4}=[level distance=1cm,sibling distance=5mm]
	
\node[node,label=south west:{$\frac{3}{4}$}](34){}
child[grow=left,<-,level distance=2.4cm]{[no edge] node[node,label=below:{$\frac{7}{8}$}](78){}child[grow=left,<-,level distance=2.4cm]{[no edge] node[node,label=below:{$\frac{15}{16}$}](1516){}}}
child[grow=right,level distance=2.4cm,dashed]{node[node,label=below:{$\frac{1}{2}$}](12){}child[grow=right,level distance=2.4cm,dashed]{node[node,label=below:{$\frac{1}{4}$}](14){}}}
child[grow=down,->,level distance=1.5cm]{node[node,label=north west:{$\frac{5}{8}$}]{}
	child[grow=south west,<-,dashed,level distance=1.2cm]{node[node,label=right:{$\frac{13}{16}$}](1316){} 
		child[grow=left,<-,level distance=2cm]{[no edge] node[node](1316a){}child[grow=left,<-,level distance=2cm]{[no edge] node[node](1316b){}}}
		child[grow=south west,->,solid,level distance=1.2cm]{node[node,label=above:{$\frac{11}{16}$}](1116){}
		child[grow=south west,<-,dashed,level distance=7.5mm]{node[node,label=north west:{$\frac{27}{32}$}](2732){}
			child[grow=left,<-,level distance=1.2cm]{[no edge] node[node](2732a){}child[grow=left,<-,level distance=1.2cm]{[no edge] node[node](2732b){}}}
			child[grow=south west,->,solid,level distance=7.5mm]{node[node,label=left:{$\frac{23}{32}$}](2332){}
			child[grow=south west,<-,dashed,level distance=1cm]{}
			child[grow=south east,->,solid,level distance=1cm]{}
			child[grow=right,->,dashed,level distance=6mm]{node[node](1){} child[grow=right,level distance=6mm]{node[node](8){}}}}}
		child[grow=south east,->,solid,level distance=1.5cm]{node[node,label=left:{$\frac{19}{32}$}]{}
			child[grow=south west,<-,dashed,level distance=1cm]{}
			child[grow=south east,->,solid,level distance=1cm]{}
			child[grow=right,->,dashed,level distance=0.2cm]{node[node](9){} child[grow=right,level distance=0.2cm]{node[node](10){}}}}
		child[grow=right,->,dashed,level distance=1cm]{node[node](2){} child[grow=right,level distance=1cm]{node[node](11){}}}}}
	child[grow=south east,->,solid,level distance=2.4cm]{node[node,label=north east:{$\frac{9}{16}$}]{}
		child[grow=south west,<-,dashed,level distance=7.5mm]{node[node,label=north west:{$\frac{25}{32}$}](2532){}
			child[grow=left,<-,level distance=0.6cm]{[no edge] node[node](2532a){}child[grow=left,<-,level distance=0.6cm]{[no edge] node[node](2532b){}}}
			child[grow=south west,->,solid,level distance=7.5mm]{node[node,label=left:{$\frac{21}{32}$}](2332){}
			child[grow=south west,<-,dashed,level distance=1cm]{}
			child[grow=south east,->,solid,level distance=1cm]{}
			child[grow=right,->,dashed,level distance=0.6cm]{node[node](3){} child[grow=right,level distance=0.6cm]{node[node](3b){}}}}}
		child[grow=south east,->,solid,level distance=1.5cm]{node[node,label=north east:{$\frac{17}{32}$}]{}
			child[grow=south west,<-,dashed,level distance=1cm]{}
			child[grow=south east,->,solid,level distance=1cm]{}
			child[grow=right,->,dashed,level distance=1cm]{node[node](4){} child[grow=right,level distance=1cm]{node[node](4b){}}}}
		child[grow=right,->,dashed,level distance=1.5cm]{node[node](6){} child[grow=right,level distance=1.5cm]{node[node](6b){}}}}
	child[grow=right,->,dashed,level distance=2.4cm]{node[node](5){} child[grow=right, level distance=2.4cm]{node[node](7){}}}
};
\draw (1516) edge[bend right=10] (78);
\draw (78) edge[bend right=10] (34);
\draw (1516) edge[bend left=10,dashed] (78);
\draw (78) edge[bend left=10,dashed] (34);
\draw (1316b) edge[bend right=10] (1316a);
\draw (1316a) edge[bend right=10] (1316);
\draw (1316b) edge[bend left=10,dashed] (1316a);
\draw (1316a) edge[bend left=10,dashed] (1316);
\draw (2732b) edge[bend right=10] (2732a);
\draw (2732a) edge[bend right=10] (2732);
\draw (2732b) edge[bend left=10,dashed] (2732a);
\draw (2732a) edge[bend left=10,dashed] (2732);
\draw (2532b) edge[bend right=10] (2532a);
\draw (2532a) edge[bend right=10] (2532);
\draw (2532b) edge[bend left=10,dashed] (2532a);
\draw (2532a) edge[bend left=10,dashed] (2532);
\draw (12) edge[loop above] (12);
\draw (14) edge[loop above] (14);
\draw (1) edge[loop above,min distance=6mm,in=55,out=125,looseness=10] (1);
\draw (2) edge[loop above] (2);
\draw (3) edge[loop above,min distance=6mm,in=55,out=125,looseness=10] (3);
\draw (3b) edge[loop above,min distance=6mm,in=55,out=125,looseness=10] (3b);
\draw (4) edge[loop above] (4);
\draw (4b) edge[loop above] (4b);
\draw (5) edge[loop above] (5);
\draw (6) edge[loop above] (6);
\draw (6b) edge[loop above] (6b);
\draw (7) edge[loop above] (7);
\draw (8) edge[loop above,min distance=6mm,in=55,out=125,looseness=10] (8);
\draw (9) edge[loop above] (9);
\draw (10) edge[loop above] (10);
\draw (11) edge[loop above] (11);
\end{tikzpicture}\end{figure}

\section{Schreier graphs of finitely generated subgroups of $H(\Z)$ and $\G$}\label{schreier}

We will build on the result from Remark~\ref{tree}. In a more general case, the comparison lemma by Baldi-Lohoué-Peyrière (Lemma~\ref{var}) implies that the existence of a regular subtree (like $\tilde{\Gamma}$) is enough to ensure transience on the Schreier graph. To obtain such a tree, we only need the assumptions of the remark inside the closed interval $[b,c]$. We will now prove a lemma that ensures transience while allowing the graph to be more complicated outside $[b,c]$. This will help us understand subgroups of $G$ for which the supports of their generators are not necessarily single intervals.

\begin{lemma}\label{algtho}
Let $f,g$ be homeomorphisms on $\R$ and assume that there exist $b<c$ such that $g(b)=b$, $f(c)=c$, $(b,c]\subset\supp(g)$ and $[b,c)\subset\supp(f)$. Assume also that there exists $s\in\R$ with $s\leq b$ such that for some $n\in\Z$, $f^n(s)\in[b,c]$. Let $H$ be the subgroup of the group of homeomorphisms on $\R$ generated by $f$ and $g$. Then the simple walk of $H$ on the Schreier graph $\Gamma$ of $H$ on the orbit $s$ is transient. 
\end{lemma}

\begin{proof}
Without loss of generality, $f(x)>x$ and $g(x)>x$ for $x\in(b,c)$ (and the end point that they do not fix). In that case clearly $n\geq0$. We will apply the comparison lemma by Baldi-Lohoué-Peyrière (Lemma~\ref{var}) with $P_1$ defined on $\Gamma$ as the kernel of the simple random walk of $H$ on $\Gamma$. In other words, $P_1(x,f(x))=P_1(x,f^{-1}(x))=P_1(x,g(x))=P_1(x,g^{-1}(x))=\frac{1}{4}$ for every $x\in\Gamma$. Let us now define $P_2$. Let $a$ be the largest fixed point of $f$ that is smaller than $b$, and $d$ the smallest fixed point of $g$ that is larger than $c$. For $x\in(a,b)$ we define $n(x)=min(n|f^n(x)\in[b,c])$. Similarly, we define for $x\in(c,d)$, $m(x)=min(m|g^{-m}\in[b,c])$. We define

\begin{equation*}\begin{minipage}{8.2cm}
$P_2(x,f(x))=\begin{cases}
\frac{1}{4} & x\in[b,c] \\
\frac{1}{4} & x\in(a,b)\mbox{ and }n(x)\mbox{ is odd}\\
\frac{3}{4} & x\in(a,b)\mbox{ and }n(x)\mbox{ is even}\\
0 & \mbox{otherwise.}\\
\end{cases}$\end{minipage}\begin{minipage}{8.3cm}
$P_2(x,f^{-1}(x))=\begin{cases}
\frac{1}{4} & x\in[b,c] \\
\frac{3}{4} & x\in(a,b)\mbox{ and }n(x)\mbox{ is odd}\\
\frac{1}{4} & x\in(a,b)\mbox{ and }n(x)\mbox{ is even}\\
0 & \mbox{otherwise.}\\
\end{cases}$\end{minipage}
\end{equation*}
\begin{equation*}\begin{minipage}{8.2cm}
$P_2(x,g(x))=\begin{cases}
\frac{1}{4} & x\in[b,c] \\
\frac{3}{4} & x\in(c,d)\mbox{ and }m(x)\mbox{ is odd}\\
\frac{1}{4} & x\in(c,d)\mbox{ and }m(x)\mbox{ is even}\\
0 & \mbox{otherwise.}\\
\end{cases}$\end{minipage}\begin{minipage}{8.3cm}
$P_2(x,g^{-1}(x))=\begin{cases}
\frac{1}{4} & x\in[b,c] \\
\frac{1}{4} & x\in(c,d)\mbox{ and }m(x)\mbox{ is odd}\\
\frac{3}{4} & x\in(c,d)\mbox{ and }m(x)\mbox{ is even}\\
0 & \mbox{otherwise.}\\
\end{cases}$\end{minipage}
\end{equation*}

Of course, we have $P_2(x,y)=0$ otherwise. This clearly defines a stochastic kernel (as the sum of probabilities at each $x$ is $1$), and it follows directly from the definition that it is symmetric. It is therefore doubly stochastic and symmetric. 

We now check that it is transient similarly to Lemma~\ref{tree-old}. Indeed, take a point $x\in[f(b),c]$ (respectively $x\in[b,g^{-1}(c)]$). Consider the subgraph $\tilde{\Gamma}(x)$ of the vertices of the form $c_nc_{n-1}\dots c_1(x)$ with $c_ic_{i-1}\dots c_1(x)\in [b,c]$ for every $i$ and $c_1\in\{f^{-1},g^{-1}\}$ (respectively $c_1\in\{g,f\}$). Equivalently to Lemma~\ref{tree-old}, $\tilde{\Gamma}(x)$ is a binary tree. Moreover, the graph $\bar{\Gamma}(x)$ defined by the vertices of the form $\tilde{c}^n(y)\in\Gamma$ with $\tilde{c}\in\{g,f^{-1}\}$, $n\in\N$ and $y\in\tilde{\Gamma}(x)$ is equivalent to the one in Lemma~\ref{tree-old}. In particular, the simple random walk on it is transient. Take any $y\in\Gamma\cap(a,d)$. Then either $f^n(y)\in[f(b),c]$ for some $n$, or $g^{-n}\in[b,g^{-1}(c)]$. In either case, there is $x$ such that $y$ belongs to $\bar{\Gamma}(x)$. By the comparison lemma by Baldi-Lohoué-Peyrière (Lemma~\ref{var}), we have $\sum_{n\in\N}\langle P_2^n\delta_y,\delta_y\rangle<\infty$. Therefore $P_2$ is transient. We apply Lemma~\ref{var} again for $P_1\geq\frac{1}{3}P_2$, which concludes the proof.
\end{proof}

Remark that with this result we can apply the comparison lemma by Baldi-Lohoué-Peyrière (Lemma~\ref{var}) to obtain transience for a random walk induced by a measure on a subgroup of the piecewise $\Pl$ group $\G$ (see Definition~\ref{tildeg}), the support of which contains two such elements and generates that subgroup as a semi-group.

For the sake of completeness, we will also consider amenability of Schreier graphs of subgroups of $\G$. A locally finite graph is called amenable if for every $\varepsilon$ there exists a finite set of vertices $S$ such that $|\partial S|/|S|<\varepsilon$ where $\partial S$ is the set of vertices adjacent to $S$. This closely mirrors F{\o}lner's criterion for amenability of groups. In particular, a finitely generated group is amenable if and only if its Cayley graph is. In his article, Savchuk~\cite{slav10} shows that the Schreier graph of the dyadic action of Thompson's group $F$ is amenable. He also mentions that it was already noted in private communication between Monod and Glasner. The amenability of the graph comes from the fact that sets with small boundary can be found in the rays (see Figure~\ref{sav}). We will prove that for finitely generated subgroups of $\G$ we can find sets quasi-isometric to rays.

\begin{remark}Consider a point $s\in\R$ and a finitely generated subgroup $H$ of the piecewise $\Pl$ group $\G$ (see Definition~\ref{tildeg}). Let $a=\sup(Hs)$. Let $S$ be a finite generating set and consider the Schreier graph $\Gamma$ defined by the action of $H$ on $Hs$. Then there is $b<a$ such that the restriction of $\Gamma$ to $(b,a)$ is a union of subgraphs quasi-isometric to rays.
\end{remark}

\begin{proof}As all elements of $H$ are continuous (when seen as functions on $\R$), they all fix $a$. Therefore they admit left germs at $a$. By definition, the germs belong to the stabiliser $St_a$ of $a$ in $PSL_2(\Z)$.
	
By Lemma~\ref{cyclic}, $St_a$ is cyclic. Let $h\in PSL_2(\Z)$ be a generator of $St_a$. Then the left germ at $a$ of any element $s_i\in S$ is equal to $h^{n_i}$ for some $n_i\in\Z$. Up to replacing $h$ with $h^{GCD(\{n_i:s_i\in S\})}$, we can assume that there exists $g\in H$ such that the left germ at $a$ of $g$ is $h$. Let $(b,a)$ be a small enough left neighbourhood such that the restrictions of all elements of $S\cup\{g\}$ to $(b,a)$ are equal to their left germs at $a$. For example, one can choose $b$ to be the largest break point of an element of $S\cup\{g\}$ that is smaller than $a$.

Consider the following equivalence relation on $Hs\cap(b,a)$: $x\sim y$ if and only if there exists $n\in\Z$ such that $h^n(x)=y$. As the restriction of $h$ to $(b,a)$ is an increasing function, an equivalence class is of the form $(h^n(x))_{n\in\N}$ for some $x\in(b,a)$. We will prove that this set is quasi-isometric to a ray (when seen as a subgraph of $\Gamma$). It is by definition of $b$ preserved by elements of $S$. Furthermore, the graph distance $d$ is bilipschitz to the standard distance $d'$ on $\N$. Indeed, on one hand, we have $d>\frac{1}{\max(|n_i|:s_i\in S)}d'$. On the other hand, $d<|g|d'$ where $|g|$ is the word length of $g$. This proves the result.
\end{proof}

This implies:

\begin{remark}\label{grapham}Consider a point $s\in\R$ and a finitely generated subgroup $H<\G$. The Schreier graph defined by the action of $H$ on $Hs$ is amenable.
\end{remark}

\section{Convergence conditions based on expected number of break points}\label{anothersuff}

The aim of this section is to describe sufficient conditions for convergence similar to Theorem~\ref{base} that do not assume leaving $C_\mu$ (which is potentially infinite). The ideas presented are similar to the arguments used in studies of measures with finite first moment on wreath products (see Kaimanovich~\cite[Theorem~3.3]{Kaimanovich1991}, Erschler~\cite[Lemma~1.1]{erschler2011}). Consider the piecewise $\Pl$ group $\G$ (see Definition~\ref{tildeg}) and a measure $\mu$ on it. We think of the measure as something that could be positive on all points of $\G$. Fix $s\in P_\Z\cup\Q$ and denote, for $g\in\G$, $A_g=\supp(C_g)$ (for $s\in\Q$, see discussion after Definition\ref{confdef} and after the proof of Lemma~\ref{log}). Take $x\in Gs$ and consider a random walk $(g_n)_{n\in\N}$ with increments $h_n$, that is $g_{n+1}=h_ng_n$. Then by (\ref{der}),

$$C_{g_n}(x)\neq C_{g_{n+1}}(x)\iff g_n(x)\in A_{h_n}.$$

In other words, $C_{g_n}(x)$ converges if and only if $g_n(x)\in A_{h_n}$ only for a finite number of values of $n$. For a fixed $n$, the probability that $g_n(x)$ belongs to $A_{h_n}$ is

$$\langle p^{*n}\delta_x,\sum_{h\in\G}\mu(h)\chi_{A_h}\rangle$$
where $p$ is the induced kernel on $Gs$. Taking the sum over $n$ we get:

\begin{lemma}\label{conv}
	Fix $\mathfrak{o}\in Gs$. For a random walk $g_n$ on $\G$ with law $\mu$, the value $C_{g_n}(\mathfrak{o})$ converges with probability $1$ if and only if
	
$$\sum_{n\in\N}\langle p^{*n}\delta_\mathfrak{o},\sum_{h\in\G}\mu(h)\chi_{A_h}\rangle<\infty$$
where $p$ is the induced kernel on $Gs$.
\end{lemma}

We define $f_\mu$ as

\begin{equation}\label{fmu}
f_\mu=\sum_{h\in\G}\mu(h)\chi_{\supp(C_h)}
\end{equation}
%
%
%
%
and show that it suffices for $f_\mu$ to be $l^1$ and $\mu$ transient :

\begin{lemma}\label{ltwo}
	Let $s\in P_\Z\cup\Q$ be fixed. Take a measure $\mu$ on $\G$ such that the induced random walk on the Schreier graph on $Gs$ is transient and $f_\mu\in l^1(Gs)$ (as defined in (\ref{fmu})). Then for a random walk $g_n$ on $\G$ with law $\mu$, the associated configuration $C_{g_n}$ converges pointwise with probability $1$.
\end{lemma}

Remark in particular that $\mathbb{E}[Br]<\infty$ implies $f_\mu\in l^1(\G)$, where $Br(g)$ is the number of break points of $g$. Indeed, for any fixed $s$, $\|f_\mu\|_1$ is the expected number of break points inside the orbit $Gs$, which is smaller than the total expected number of break points. This is, of course, also true for measures on $H(\Z)$ as $H(\Z)\leq\G$.

\begin{proof}
Fix a point $\mathfrak{o}$ in the Schreier graph on $Gs$. We denote by $p$ the induced kernel on $Gs$ and write $f=f_\mu$. We have

\begin{equation}\label{ltwosum}
\sum_{n\in\N}\langle p^{*n}\delta_\mathfrak{o},f\rangle=\sum_{n\in\N}\sum_{x\in Gs}p^{*n}(\mathfrak{o},x)f(x)=\sum_{x\in Gs}f(x)\sum_{n\in\N}p^{*n}(\mathfrak{o},x)
\end{equation}
where we will have the right to interchange the order of summation if we prove that the right-hand side is finite. We write $p^{*n}(\mathfrak{o},x)=\check{p}^{*n}(x,\mathfrak{o})$ where $\check{p}$ is the inverse kernel of $p$. Let $\check{P}(x,y)$ be the probability that a random walk (with law $\check{p}$) starting at $x$ visits $y$ at least once. Then $\sum_{n\in\N}\check{p}^{*n}(x,y)=\check{P}(x,y)\sum_{n\in\N}\check{p}^{*n}(y,y)$. Indeed, $\sum_{n\in\N}\check{p}^{*n}(x,y)$ is the expected number of visits of $y$ of a walk starting at $x$ and random walk that starts from $x$ and visits $y$ exactly $k$ times is the same as the concatenation of a walk that goes from $x$ to $y$ and a walk that starts from $y$ and visits it $k$ times. Thus

\begin{equation}\label{ltwoinv}
\sum_{n\in\N}p^{*n}(\mathfrak{o},x)=\sum_{n\in\N}\check{p}^{*n}(x,\mathfrak{o})=\check{P}(x,\mathfrak{o})\sum_{n\in\N}\check{p}^{*n}(\mathfrak{o},\mathfrak{o})\leq\sum_{n\in\N}\check{p}^{*n}(\mathfrak{o},\mathfrak{o}).
\end{equation}

Then if we denote $c(p,\mathfrak{o})=\sum_{n\in\N}p^{*n}(\mathfrak{o},\mathfrak{o})$,

\begin{equation}\label{ltwofin}
\sum_{x\in Gs}f(x)\sum_{n\in\N}p^{*n}(\mathfrak{o},x)\leq c(p,\mathfrak{o})\|f\|_1<\infty.
\end{equation}

Applying Lemma~\ref{conv} we obtain the result.
\end{proof}

Combining this result with the result of Lemma~\ref{algtho} which gives transience of the induced random walk on $Gs$ under certain conditions, we obtain:

\begin{lemma}\label{algone}
	Consider the piecewise $\Pl$ group $\G$ (see Definition~\ref{tildeg}). Let $H$ be a subgroup of $\G$. Assume that there exist $b<c$ such that $g(b)=b$, $f(c)=c$, $(b,c]\subset\supp(g)$ and $[b,c)\subset\supp(f)$ for some $f,g\in H$ (see Figure~\ref{alto} on page~\pageref{alto}). Assume also that there exists $s\in P_\Z\cup\Q$ and $\varepsilon_s>0$ with $s\leq b$ such that for some $n\in\Z$, $f^n(s)\in[b,c]$, and also $g(s-\varepsilon)=s-\varepsilon$ and $g(s+\varepsilon)\neq s+\varepsilon$ for every $0<\varepsilon\leq\varepsilon_s$. Then for any $\mu$ on $H$ with finite first break moment ($\mathbb{E}[Br]<\infty$) such that $\supp(\mu)$ generates $H$ as a semigroup, the Poisson boundary of $\mu$ on $H$ is non-trivial. 
\end{lemma}

\begin{proof}
	By Lemma~\ref{algtho}, the simple random walk on the Schirer graph of $s$ by $\langle f,g\rangle$ is transient. By the comparison lemma by Baldi-Lohoué-Peyrière (Lemma~\ref{var}), as the support of $\mu$ generates $H$ as a semigroup, the random walk by $\mu$ on the Schreier graph of $s$ is then transient. Applying Lemma~\ref{ltwo}, the associated configurations converge as $\mu$ has finite first break moment. However, by hypothesis on $s$, $g(s)=s$ and $C_g(s)\neq 0$. Therefore, as $g\in H$, the limit configuration cannot be singular. Thus the Poisson boundary of $\mu$ on $H$ is non-trivial.	
\end{proof}

For finitely generated subgroups of $\G$, from Lemma~\ref{ltwo} we have:

\begin{remark}\label{brfin}
	The amount of break points is subadditive in relation to multiplication. In particular, if a measure $\mu$ has finite first moment, then it has finite first break moment.
\end{remark}

\begin{cor}\label{firstfin}
Consider a measure $\mu$ on $\G$, the support of which generates a finitely generated subgroup, and such that $\mu$ has a finite first moment on that subgroup. Assume that there exists $s\in P_\Z$ such that the random walk on the Schreier graph on $Gs$ of this subgroup is transient. Then, for almost all random walks on $\G$ with law $\mu$, the associated configuration converges pointwise.
\end{cor}

\begin{proof}
Follows from Remark~\ref{brfin} and Lemma~\ref{ltwo}.\end{proof}

In such cases it is enough to prove that the associated limit configuration is not always the same, which can require case-specific arguments. We already have it in the case of Thompson's group:

\begin{proof}[Proof of Corollary~\ref{finfirstthomp}]Fix $s\in P_\Z$ and consider the action $\mathfrak{a}_s$ of Thompson's group $F$ on $\R$ as defined in Section~\ref{thompsect}. Take a measure $\mu$ on $F$ that generates it as a semigroup. From Lemma~\ref{tree-old} and the comparison lemma by Baldi-Lohoué-Peyrière (Lemma~\ref{var}) the walk $\mu$ induces on the orbit of $s$ is transient. Applying Corollary~\ref{firstfin} this implies that the associated configuration stabilises, and by Lemma~\ref{nostable}, it cannot always converge towards the same point. Therefore the Poisson boundary of $\mu$ is not trivial.\end{proof}

We remark that arguments similar to the ones in this section can also be made for the action of Thompson's group considered in Kaimanovich's article~\cite{kaimanovichthompson}. 

In a more general case, we can use the stronger result by Varopoulos of the comparison Lemma~\ref{var} in order to prove that if the transient walk diverges quickly enough, we can also have the result for $f_\mu\in l^2(Gs)$ (and not necessarily in $l^1$):

\begin{lemma}\label{ltwoforreal}
Fix $s\in P_\Z$. Consider a measure $\mu_0$ such that $\tilde{f}=f_{\mu_0}\in l^2(Gs)$. Consider $\lambda$ on $H_s$ such that $\sum_{n\in\N}\langle\lambda^{*n}\tilde{f},\tilde{f}\rangle<\infty$. Let $\mu=\varepsilon\lambda+(1-\varepsilon)\mu_0$ with $0<\varepsilon<1$. Then for almost all random walks on $G$ with law $\mu$, the associated configuration converges pointwise.
\end{lemma}

\begin{proof}
Clearly, $f_\mu=(1-\varepsilon)\tilde{f}$. Then by the comparison Lemma~\ref{var} we get:

$$\sum_{n\in\N}\langle\mu^{*n}f_\mu,f_\mu\rangle<\frac{1}{\varepsilon(1-\varepsilon)^2}\sum_{n\in\N}\langle\lambda^{*n}\tilde{f},\tilde{f}\rangle<\infty.$$

Denote $f=f_\mu$. Consider $x\in P_\Z$ such that it is possible for the value of the associated configuration at $x$ to change. In other words, there is $n_0\in\N$ and $y\in P_\Z$ such that $x\in\supp(\mu^{*n_0})y$ and $f(y)>0$. Denote by $p$ the probability to reach $x$ from $y$. Then $\sum_{n\in\N}\langle\mu^{*n}\delta_y,f\rangle>p\sum_{n\in\N}\langle\mu^{*n+n_0}\delta_x,f\rangle$. In particular, if the first is finite, so is the second. However, we clearly have $\sum_{n\in\N}\langle\mu^{*n}\delta_y,f\rangle<\frac{1}{f(y)}\sum_{n\in\N}\langle\mu^{*n}f,f\rangle$ which concludes the proof.
\end{proof}

In particular, if for any $s$ all associated configurations cannot be stable by all the elements of $\langle\supp(\mu)\rangle$, we obtain a non-trivial boundary.

\begin{cor}
Fix $s\in P_\Z$. Consider a measure $\mu_0$ such that $h_s\in\supp(\mu_0)^{*n_0}$ for some $n_0$ and $\tilde{f}=f_{\mu_0}\in l^2(Gs)$. Consider $\lambda$ on $H_s$ such that $\sum_{n\in\N}\langle\lambda^{*n}\tilde{f},\tilde{f}\rangle<\infty$. Let $\mu=\varepsilon\lambda+(1-\varepsilon)\mu_0$ with $0<\varepsilon<1$. Then the Poisson boundary of $\mu$ on the subgroup generated by its support is non-trivial.
\end{cor}
\begin{proof}
	Follows from Lemma~\ref{ltwoforreal} and Lemma~\ref{nostable}.
\end{proof}

Remark that there always exists a symmetric measure $\lambda$ satisfying those assumptions as $\A\subset H_s$ ($\A$ was defined in (\ref{agrp})).

\begin{figure}
	\centering
	\begin{minipage}{8cm}\centering\caption{Graphs of $f$ and $g$ and positions of $b$ and $c$}\label{alto}
		\begin{tikzpicture}
		\begin{axis}[xmin=-4,xmax=4,ymin=-4,ymax=4,axis lines = middle, legend pos = south west,xtick={-10},ytick={17}]
		\addplot[domain=-3.8:3.8,color=black]{x};
		\addlegendentry{$Id$}
		\addplot[color=blue,samples=100,domain=0:2.5,restrict y to domain=-4:4,dashed,thick]{(2*x+3)/(x+2)};
		\addlegendentry{$f$}
		\addplot[samples=100,domain=0:2.5,restrict y to domain=-4:4,densely dotted,thick]{(4*x-1)/x};
		\addlegendentry{$g$}
		\node[label={-1:{$b$}},circle,fill,inner sep=1pt] at (axis cs:0.268,0.268) {};
		\node[label={110:{$c$}},circle,fill,inner sep=1pt] at (axis cs:1.732,1.732) {};
		\end{axis}
		\end{tikzpicture}
	\end{minipage}
	\begin{minipage}{8cm}\centering\caption{Graphs of $f$ and $g$ in $(a,b')$}\label{alte}
		\begin{tikzpicture}
		\begin{axis}[xmin=-4,xmax=4,ymin=-4,ymax=4,axis lines = middle, legend pos = south west,xtick={-10},ytick={17}]
		\addplot[domain=-3.8:3.8,color=black]{x};
		\addlegendentry{$Id$}
		\addplot[samples=100,domain=0.4365:1.732,restrict y to domain=-4:4,densely dotted,thick]{(2*x+3)/(x+2)};
		\addlegendentry{$f$}
		\addplot[color=blue,samples=100,domain=0.382:2.618,restrict y to domain=-4:4,dashed,thick]{(3*x-1)/x};
		\addlegendentry{$g$}
		\addplot[samples=100,domain=0.382:0.4365,restrict y to domain=-4:4,densely dotted,thick]{(8*x-3)/(3*x-1)};
		\node[label={-2:{$a$}},circle,fill,inner sep=1pt] at (axis cs:0.382,0.382) {};
		\node[label={-30:{$b$}},circle,fill,inner sep=1pt] at (axis cs:1.732,1.732) {};
		\node[label={-30:{$b'$}},circle,fill,inner sep=1pt] at (axis cs:2.618,2.618) {};
		\end{axis}
		\end{tikzpicture}
	\end{minipage}
\end{figure}

\section{An algebraic lemma and proof of the main result}\label{algsec}

Consider the piecewise $\Pl$ group $\G$ (see Definition~\ref{tildeg}). Take a subgroup $H$ of $\G$. In Lemma~\ref{algone} we proved that if there are $f,g\in H$ and $b,c,s\in\R$ that satisfy certain assumptions, for every measure $\mu$ on $H$ the support of which generates $H$ as a semigroup and that has finite first break moment $\mathbb{E}[Br]$, $(H,\mu)$ has non-trivial Poisson boundary. To prove the main result (Theorem~\ref{main}) we will study subgroups that do not contain elements satisfying those assumptions.

\begin{lemma}\label{algthree}
Let $H=\langle h_1,\dots,h_k\rangle$ be a finitely generated subgroup of $\G$. Then either $H$ is solvable, or the assumptions of Lemma~\ref{algone} are satisfied for some $f,g\in H$, $b,c,s\in\R$.
\end{lemma}

We recall that for $f\in\G$, and $a,b\in\R$ such that $f(a)=a$ and $f(b)=b$, we defined (see Definition \ref{restr}) $f\restriction_{(a,b)}\in\G$ by $f\restriction_{(a,b)}(x)=f(x)$ for $x\in(a,b)$ and $x$ otherwise.

\begin{proof}
	
We first check that with the appropriate assumptions on $(f,g,b,c)$, $s$ always exists:

\begin{lemma}\label{algtwo}
	Let $H$ be a subgroup of $\G$. Assume that there exist $b<c$ such that $g(b)=b$, $f(c)=c$, $(b,c]\subset\supp(g)$ and $[b,c)\subset\supp(f)$ for some $f,g\in H$. Then there exist $f',g',b',c'$ and $s$ that satisfy the assumptions of Lemma~\ref{algone}.
\end{lemma}

The assumptions of the lemma are illustrated in Figure~\ref{alto}. Recall that we defined $\supp(f)=\{x\in\R:f(x)\neq x\}$.

\begin{proof}
Without loss of generality assume that $b$ is minimal among all $b$ for which there exists $c$ such that either $(f,g,b,c)$ or $(g,f,b,c)$ satisfy the assumptions of this lemma. We can assume without loss of generality that $f(x)>x$ and $g(x)>x$ for $x\in(b,c)$ (otherwise, we can replace either or both with their inverse). Let $a$ be the largest fixed point of $f$ that is smaller than $b$.
	
By minimality of $b$ we clearly have that $g(a)=a$. The stabiliser $St_a$ of $a$ in $\Pl$ is cyclic by Lemma~\ref{cyclic}. Therefore there exist $k$ and $l$ such that $f^k(x)=g^l(x)$ for $x\in(a,a+\varepsilon)$ for some $\varepsilon>0$. Take $(f',g')=(f,f^{-k}g^l)$. By our assumption, $f^k$ and $g^l$ are strictly greater then the identity function in $(b,c)$. As they are continuous and each fixes an end of the interval, by the mean values theorem there exists $b'\in(b,c)$ such that $f^k(b')=g^l(b')$. Then $(f',g')$ and $(b',c)$ satisfy the assumptions of this lemma. Furthermore, $f^{-k}g^l$ is the identity in a small enough right neighbourhood of $a$, which implies that there exists an element $s$ that satisfies the assumptions of Lemma~\ref{algone}.
\end{proof}

We now assume that the assumptions of Lemma~\ref{algone}, and therefore also the assumptions of Lemma~\ref{algtwo}, are not satisfied by any couple of elements in $H$. We will prove that $H$ is solvable. For any element in $g\in\G$, its support $\supp(g)$ is a finite union of (not necessarily finite) open intervals. The intervals in the support of $h_i$ we denote $I^i_j=(a_i^j,b_i^j)$ for $j<r_i$ where $r_i$ is the number of intervals in the support of $h_i$. In terms of those intervals, the negation of Lemma~\ref{algtwo} means that for every $(i,j)$ and $(i',j')$, either $I^i_j\cap I^{i'}_{j'}=\emptyset$, or $I^i_j\subset I^{i'}_{j'}$, or $I^{i'}_{j'}\subset I^i_j$. We further check that if the inclusion is strict, it must be strict at both extremities. Specifically:

\begin{lemma}\label{algbonus}
Let $H$ be a subgroup of $\G$. Assume that there exist $a<b<b'\in\R\cup\{-\infty\}$ such that $f(a)=g(a)=a$, $f(b)=b$, $g(b')=b'$, $(a,b)\subset\supp(f)$ and $(a,b')\subset\supp(g)$ for some $f,g\in H$ (see Figure~\ref{alte}). Then the assumptions of Lemma~\ref{algtwo} are satisfied by some elements of the group. 	
\end{lemma}

\begin{proof}
In a small enough right neighbourhood of $a$ there are no break points of $f$ and $g$. Let $c$ be a point in that neighbourhood. Clearly, $a<c<b$. Without loss of generality, we can assume that $f(x)>x$ for $x\in(a,b)$, and idem for $g$ (otherwise, we can replace them with their inverse). For some $k\in\N$, $f^{-k}(b)<c$. Denote $g'=f^{-k}gf^k$. Consider the elements $g'$ and $g^{-1}g'$. As the stabiliser of $a$ in $\Pl$ is cyclic (by Lemma~\ref{cyclic}), $g^{-1}g'(x)=x$ for $x\in(a,f^{-k}(c))$. However, $g^{-1}g'(x)=g^{-1}(x)$ for $x\in(f^{-k}(b),b)$, and in particular $g^{-1}g'(x)\neq x$ in that interval. Let $c'$ be the largest fixed point of $g^{-1}g'$ that is smaller than $f^{-k}(b)$. Consider now $g'$. It is the conjugate of $g$, therefore it is different from the identity in $(a,f^{-k}(b))$ and fixes $f^{-k}(b)<c$. Clearly, $c'<f^{-k}(b)$. Then $g',g^{-1}g'$ and $c',f^{-k}(b)$ satisfy the assumptions of Lemma~\ref{algtwo}. Observe that the same arguments can be used for two elements with supports $(a,b)$ and $(a',b)$ with $a\neq a'$.
\end{proof}

Consider the natural extension of the action of $\G$ on $\R\cup\{+\infty,-\infty\}$, which is that every element of $\G$ fixes both $-\infty$ and $+\infty$. We make the convention that $+\infty$ is considered to be a break point of $f\in\G$ if and only if for every $M\in\R$ there is $x>M$ such that $f(x)\neq x$ (and idem for $-\infty$). In other words, if the support of an element is equal to an interval $(a,b)$, $a$ and $b$ are break points even if one or both are infinite. We now prove that $H$ is solvable by induction on the number of different orbits of $H$ on $\R\cup\{\pm\infty\}$ that contain non-trivial break points of elements of $H$. Remark that the number of orbits of $H$ that contain non-trivial break points of elements of $H$ is the same as the number of orbits that contain non-trivial break points of $h_1,\dots,h_k$. In particular, it is finite.

Consider all maximal (for inclusion) intervals $I^i_j$ over all couples $(i,j)$. We denote them $I_1,I_2,\dots,I_n$. By our hypothesis we have that they do not intersect each other. We denote $h_i^j=h_i\restriction_{I_j}$ and $H_j=\langle h_1^j,h_2^j,\dots,h_k^j\rangle$ for every $j<n$. As the intervals $I_j$ do not intersect each other, $H$ is a subgroup of the Cartesian product of $H_j$: 

\begin{equation}\label{maxint}H\leq\prod_{j=1}^n H_j.\end{equation}

Moreover, for every $j$, the amount of orbits with non-trivial break points of $H_j$ is not greater than that of $H$. Indeed, the orbits with break points of $H_j$ inside $I_j$ coincide with those of $H$, and it has only two other orbits containing break points, which are the singletons containing the end points of $I_j$. We just need to prove that $H$ has at least two other orbits containing non-trivial break points. If $I_j=I_{i'}^{j'}$, then the supremum and infimum of the support of $h_{i'}$ are break points, and by definition of $I_j$ their orbits by $H$ do not intersect the interior of $I_j$. The convention we chose assures that our arguments are also correct if one or both of the end points is infinite. It is thus sufficient to prove the induction step for $H_j$ for every $j$. Therefore without loss of generality we can assume $n=1$. Remark that in this case the end points of $I_1$ are both non-trivial break points, and both clearly have trivial orbits.

We denote $(a,b)=I=I_1$. Consider the germs $g_i\in St_a$ of $h_i$ at a right neighbourhood of $a$. As $St_a$ is cyclic, there exist $m_i\in\Z$ such that $\prod_i g_i^{m_i}$ generates a subgroup of $St_a$ that contains $g_i$ for all $i$. Specifically, the image in $\Z$ of this product is the greatest common divisor of the images in $\Z$ of $g_i$. We denote $h=\prod_i h_i^{m_i}$ and let, for every $i$, $n_i$ satisfy $(\prod_i g_i^{m_i})^{n_i}=g_i$. For every $i\leq k$, we consider $h'_i=h_ih^{-n_i}$.

Clearly, $H=\langle h,h'_1,h'_2,\dots,h'_k\rangle$, and there exists $\varepsilon$ such that for every $i$, $\supp(h'_i)\subset(a+\varepsilon,b-\varepsilon)$ (as the assumptions of Lemma~\ref{algbonus} are not satisfied by $h,h'_i$). Consider the set of $h^{-l}h'_ih^l$ for $i<k,l\in\Z$ and their supports. They are all elements of $H$. Furthermore, there is a power $n$ such that $h^n(a+\varepsilon)>b-\varepsilon$. Therefore, for every point $x\in(a,b)$, the number of elements of that set that contain $x$ in their support is finite. Considering the intervals that define those supports, we can therefore choose a maximal one (for the inclusion). Let $x_0$ be the lower bound of a maximal interval. By our assumption, $x_0$ is then not contained in the support of any of those elements, and neither is $x_l=h^l(x_0)$ for $l\in\Z$. We denote ${h'}_i^j=h^jh'_ih^{-j}\restriction(x_0,x_1)$. For $i<k$, let $J_i$ be the set of $j\in\Z$ such that ${h'}_i^j\neq Id$. Then $H$ is a subgroup of

\begin{equation}\label{wreath}
\left\langle h,\bigcup_{i<k}\bigcup_{j\in J_i}{h'}_i^j\right\rangle\cong\langle h\rangle\wr\left\langle\bigcup_{i<k}\bigcup_{j\in J_i}{h'}_i^j\right\rangle.
\end{equation}

For a group $F$, $\Z\wr F$ denotes the wreath product of $\Z$ on $F$. It is a group, the elements of which are pairs $(n,f)$ with $n\in\Z$ and $f\in\prod_{k\in\Z}F$ with finite support. The group multiplication is defined as $(n,f)(n',f')=(n+n',T^{n'}f+f')$, where $T^{n'}f(k)=f(k-n')$. It is a well known property of wreath products that if $F$ is solvable, so is $\Z\wr F$.

Denote $H'=\langle\bigcup_{i<k}\bigcup_{j\in J_i}{h'}_i^j\rangle$. The non-trivial break points and supports of ${h'}_i^j$ are contained in $(x_0,x_1)$, and they fix that interval. Therefore the orbits that contain those break points are the same in relation to $\langle h,H'\rangle$ and to $H'$. On the other hand, $\langle h,H'\rangle$ and $H$ act the same way locally, which means that they have the same orbits. Those two facts imply that $H'$ has at least two less orbits containing non-trivial break points than $H$ (as it does not have non-trivial break points in the orbits of the end points of $I$). That group also does not contain elements that satisfy the assumptions of Lemma~\ref{algtwo}. Indeed, assume that there are two words on $\bigcup_{i<k}\bigcup_{j\in J_i}{h'}_i^j$ and $a,b\in\R$ that satisfy those assumptions. Their supports are also contained in $(x_0,x_1)$, therefore so are $a$ and $b$. Then the same words in $\bigcup_{i<k}\bigcup_{j\in J_i}h'_i$ are equal inside $(a,b)$, and they satisfy the conditions of Lemma~\ref{algtwo}. However, $h'_i$ are elements of $H$ and this is contradictory to our assumptions.

This provides the induction step. The induction basis is the trivial group, which is solvable. Therefore $H$ is solvable.
\end{proof}

We can now prove the main result, that is that for any subgroup $H$ of $H(\Z)$ which is not locally solvable and any measure $\mu$ on $H$ such that the support of $\mu$ generates $H$ as a semigroup and has finite first break moment $\mathbb{E}[Br]$, the Poisson boundary of $(H,\mu)$ is non-trivial.

\begin{proof}[Proof of Theorem~\ref{main}]Fix $H$ and take $\mu$ on $H$ with finite first break moment and the support of which generates $H$ as a semigroup. We distinguish two cases.
	
	Assume first that there exist $f,g\in H$ and $b,c,s\in\R$ that satisfy the assumptions of Lemma~\ref{algone}. By the result of the lemma, the Poisson boundary of $(H,\mu)$ is non-trivial.
	
	We now assume that no such $f,g,b,c,s$ exist and will prove that $H$ is locally solvable. Any finitely generated subgroup $\widetilde{H}$ of $H$ clearly also does not contain such $f$ and $g$ for any $b,c,s\in\R$. Furthermore, $H(\Z)$ is a subgroup of the piecewise $\Pl$ group $\G$ (see Definition~\ref{tildeg}), and thus $\widetilde{H}$ is a subgroup of $\G$. Therefore by Lemma~\ref{algthree} we obtain that $\widetilde{H}$ is solvable, which proves that $H$ is locally solvable.\end{proof}

\section{A remark on the case of finite $1-\varepsilon$ moment}\label{last}

%

Remark that in the proof of Lemma~\ref{algthree}, for a finitely generated subgroup that does not satisfy the assumptions of Lemma~\ref{algone} we obtained more than it being solvable. If the subgroup is also non-abelian, we have proven that it contains a wreath product of $\Z$ with another subgroup (see (\ref{wreath})). In particular, it is not virtually nilpotent, which implies (as it is finitely generated) that there exists a measure on it with non-trivial boundary by a recent result of Frisch-Hartman-Tamuz-Vahidi-Ferdowski~\cite{choquet-deny}. Furthermore, it is known that
on the wreath products $\Z\wr\Z$ it is possible to obtain a measure with finite $1-\varepsilon$ moment and non-trivial Poisson boundary for every $\varepsilon>0$ (see Lemma~\ref{wreathnontriv} and discussion before and after it). The same arguments can be used in $\G$:

\begin{lemma}\label{mineps}
	For every finitely generated subgroup $H=\langle h_1,\dots,h_k\rangle$ of $\G$ that is not abelian and every $\varepsilon>0$ there exists a symmetric non-degenerate measure $\mu$ on $H$ with non-trivial Poisson boundary such that $\int_H |g|^{1-\varepsilon}d\mu(g)<\infty$, where $|g|$ is the word length of $g$.
\end{lemma}

We recall that every measure on an abelian group has trivial Poisson boundary (see Blackwell~\cite{blackwell1955}, Choquet-Deny~\cite{ChoquetDeny}).

\begin{proof}
As there is always a non-degenerate symmetric measure with finite first moment, we can assume that the assumptions of Lemma~\ref{algone} are not satisfied in $H$. We will use the results on the structure of $H$ seen in the proof of Lemma~\ref{algthree}. It is shown (see (\ref{maxint})) that $H$ is a subgroup of a Cartesian product $\prod_{j=1}^n H_j$. Specifically, there exist disjoint intervals $I_1,I_2,\dots,I_n$ such that the supports of elements of $H$ are included in the union of those intervals. Taking $h_i^j=h_i\restriction_{I_j}$ to be the restriction on one of those intervals (as defined in Definition \ref{restr}), the group $H_j$ is then equal to $\langle h_1^j,h_2^j,\dots,h_k^j\rangle$. For any $j$, consider the composition of the projection of $\prod_{j=1}^n H_j$ onto $H_j$ and the inclusion of $H$ in $\prod_{j=1}^n H_j$. Then $H_j$ is the quotient of $\prod_{j=1}^n H_j$ by the kernel of this composition, which is equal to $\{h\in\prod_{j=1}^n H_j,h\restriction_{I_j}\equiv0\}$.

We can therefore separately define measures on $H_j$ and on the kernel, and the Poisson boundary of their sum would have the Poisson boundary of the measure on $H_j$ as a quotient. In particular, it suffices to show that for some $j$ we can construct a measure on $H_j$ with non-trivial boundary satisfying the conditions of the lemma. As $H$ is non-abelian, so is at least one $H_j$. Without loss of generality, let that be $H_1$. In the proof of Lemma~\ref{algthree} we have shown (see (\ref{wreath})) that in $H_1$ there are elements $h^1$ and ${h^1}'_j$ for $j=1,2,\dots,k$ such that $H_1=\langle{h^1},{h^1}'_1,{h^1}'_2,\dots,{h^1}'_k\rangle$ and is isomorphic to a subgroup of the wreath product of $h^1$ on a group $H'$ defined by the rest of the elements. Remark that $H_1$ not being abelian implies that $H'$ is not trivial. Furthermore, by taking the group morphism of $H_1$ into $\Z\wr H'$, we see that the image of $h^1$ is the generator $(1,0)$ of the active group, while for every $j$, the image of ${h^1}'_j$ is of the form $(0,f_j)$ where $f_j$ has finite support. The following result is essentially due to Kaimanovich and Vershik~\cite[Proposition~6.1]{kaimpoisson},\cite[Theorem~1.3]{Kai83}, and has been studied in a more general context by Bartholdi and Erschler~\cite{Bartholdi2017}:

\begin{lemma}\label{wreathnontriv}
	Consider the wreath product $\Z\wr H'$ where $H'$ is not trivial, and let $\mu$ be a measure on it such that the projection of $\mu$ on $\Z$ gives a transient walk and the projection of $\mu$ on ${H'}^\Z$ is finitary and non-trivial. Then the Poisson boundary of $\mu$ is not trivial.
\end{lemma}

In the article of Kaimanovich and Vershik, it is assumed that the measure is finitary, and the acting group is $\Z^k$ for $k\geq3$, which assures transience. The proof remains unchanged with our assumptions. Remark that those results have also been generalised in the case of a measure with finite first moment that is transient on the active group, see Kaimanovich~\cite[Theorem~3.3]{Kaimanovich1991},\cite[Theorem~3.6.6]{Kaimanovich2007PoissonBO}, Erschler~\cite[Lemma~1.1]{erschler2011}.

\begin{proof}
Take a random walk $(g_n)_{n\in\N}$ on $\Z\wr H'$ with law $\mu$. Let $p$ be the projection of the wreath product onto the factor isomorphic to $H'$ that has index $0$ in ${H'}^\Z$. By the assumptions of the lemma, $p(h_n)$ stabilises, and is not almost always the same. This provides a non-trivial quotient of the Poisson boundary of $\mu$.
\end{proof}

All that is left is constructing a measure that verifies the assumptions of Lemma \ref{wreathnontriv}. Consider a symmetric measure $\mu_1$ on $\langle h^1\rangle$ that has finite $1-\varepsilon$ moment and is transient. Let $\mu_2$ be defined by being symmetric and by $\mu_2({h^1}'_j)=\frac{1}{2k}$ for every $j$. Then $\mu=\frac{1}{2}(\mu_1+\mu_2)$ is a measure on $H_1$ with non-trivial Poisson boundary.
\end{proof}

\bibliographystyle{plain}
\bibliography{hz,jus}

\end{document}